\documentclass[onefignum,onetabnum]{siamart220329}

\usepackage{amsmath}%
\usepackage{amsfonts}%
\usepackage{amssymb}%
\usepackage{commath}
\usepackage{graphicx}
\usepackage{dsfont}
\usepackage{pdfcomment}
\usepackage{tikz}
\usepackage{tikz-3dplot}
\usepackage{soul}
\usepackage{enumerate}
\usepackage{cancel}

\usepackage{hyperref}
\usepackage{cleveref}
\usepackage{esvect}
\usepackage{subcaption} % for use of sub-figures

\DeclareMathOperator{\area}{Area}

\def\D{\mathbb{D}}
\def\H{\mathbb{H}}
\def\R{\mathbb{R}}

\usepackage{array}
\newcolumntype{C}[1]{>{\centering\let\newline\\\arraybackslash\hspace{0pt}}m{#1}}

\newcommand\restr[2]{{% we make the whole thing an ordinary symbol
  \left.\kern-\nulldelimiterspace % automatically resize the bar with \right
  #1 % the function
  \vphantom{\big|} % pretend it's a little taller at normal size
  \right|_{#2} % this is the delimiter
  }}

\graphicspath{ {/images/} }

\usepackage[ruled,linesnumbered,vlined,algo2e,resetcount]{algorithm2e}
\newsiamremark{remark}{Remark}
\newsiamremark{hypothesis}{Hypothesis}
\crefname{hypothesis}{Hypothesis}{Hypotheses}
\newsiamthm{claim}{Claim}

\headers{Numerical Method for Harmonic Maps Between Hyperbolic Surfaces}{Zhipeng Zhu, Wai Yeung Lam, and Lok Ming Lui}

\title{A Structure-Preserving Numerical Method for Harmonic Maps Between High-genus Surfaces
\thanks{Submitted to the editors DATE.
}}

\author{Zhipeng Zhu\thanks{
              Department of Mathematics, The Chinese University of Hong Kong
              (\email{zpzhu@math.cuhk.edu.hk}).}
           \and
           Wai Yeung Lam\thanks{
              Department of Mathematics, University of Luxembourg
              (\email{wyeunglam@gmail.com}).}
           \and
           Lok Ming Lui\thanks{
              Department of Mathematics, The Chinese University of Hong Kong
              (\email{lmlui@math.cuhk.edu.hk}).}
}

\ifpdf
\hypersetup{
pdftitle={A Structure-Preserving Numerical Method for Harmonic Maps Between High-genus Surfaces},
  pdfauthor={Zhipeng Zhu, Wai Yeung Lam, Lok Ming Lui}
}
\fi

\begin{document}

\maketitle

\begin{abstract}
Motivated by geometry processing for surfaces with non-trivial topology, we study discrete harmonic maps between closed surfaces of genus at least two. Harmonic maps provide a natural framework for comparing surfaces by minimizing distortion. Unlike conformal or isometric maps—which may not exist between surfaces with different geometries—harmonic maps always exist within a fixed homotopy class and yield optimal homeomorphisms when the target surface has negative curvature.
We develop a structure-preserving algorithm to compute harmonic maps from a triangulated surface to a reference hyperbolic surface. The method minimizes Dirichlet energy over geodesic realizations of the surface graph into the target hyperbolic surface in the homotopy class of a homeomorphism. A central feature of our framework is the use of canonical edge weights derived from the hyperbolic metric, which generalize the classical cotangent weights from the Euclidean setting. These weights preserve injectivity and ensure that isometries remain harmonic in the discrete theory, reflecting their classical behavior.
\end{abstract}

% REQUIRED

% REQUIRED

\section{Introduction}

Constructing optimal homeomorphisms between surfaces is a fundamental problem in geometry processing, with applications in surface registration, shape analysis, and computer graphics. When surfaces have non-trivial topology, such as closed surfaces of genus $g \geq 2$, the problem becomes significantly more challenging due to the lack of canonical parameterizations and the variability of their geometric structures.

Harmonic maps offer a natural and robust framework for addressing this challenge. Given two Riemannian surfaces of the same topological type, a harmonic map is a critical point of the Dirichlet energy within a fixed homotopy class. Such maps always exist when the target surface has non-positive curvature and, when fixed in the homotopy class of a homeomorphism, provide diffeomorphisms that minimize geometric distortion. Unlike conformal or isometric maps that may not exist, harmonic maps arise naturally as minimizers of an energy functional, making them attractive for practical applications.

Although the theory of harmonic maps is well established in the smooth setting, its discrete counterpart requires careful attention to discretization choices. Naive discretizations of the Dirichlet energy may fail to preserve important structural properties such as injectivity or stability under geometric deformation. This motivates the development of \emph{structure-preserving discretizations}—discrete analogues that faithfully reflect the geometry of the surfaces involved and ensure that the resulting numerical maps behave consistently with the smooth theory.

In this work, we develop a structure-preserving numerical method for computing discrete harmonic maps between hyperbolic surfaces. Given a geodesically triangulated hyperbolic surface $S_1=(V,E,F)$, our method computes a discrete harmonic map to a target hyperbolic surface $S_2$, which is identified with $S_1$ via a homeomorphism. Here the discrete harmonic map is defined as the minimizer of the Dirichlet energy over all geodesic realizations of the 1-skeleton graph of the triangulation $f:(V,E)\to S_2$ within a homotopy class induced from the homeomorphism. With a prescribed edge weights $c:E\to \R_{>0}$, the discrete Dirichlet energy of a geodesic realization $f$ is defined as
\begin{equation*}
    D_c(f) = \frac{1}{2}\sum_{ij} c_{ij} \ell_{ij}^2,
\end{equation*}
where each edge $ij$ is realized as a geodesic with hyperbolic length $\ell_{ij}$. Following a theorem of Colin de Verdière~\cite{CdV1991}, we can show that there is a unique minimizer of $D_c(f)$ in the given homotopy class, which is necessarily an embedding. 

A novelty of our framework is the use of \emph{canonical edge weights}, which capture the geometry of $S_1$ and generalize the classical cotangent weights from the Euclidean setting \cite{pinkall1993computing} to the hyperbolic case \cite{Lam2024harmonic}. The canonical edge weights satisfy two essential properties: they are positive when the input triangulation is Delaunay, ensuring local \emph{injectivity} of the resulting map, and they preserve \emph{harmonicity of isometries}. These properties make the weights particularly suitable for geometry-aware computations and ensure that the discretization preserves key structural features of the smooth theory.

We show that the discrete Dirichlet energy defined using these weights is a proper and differentiable functional on the space of geodesic realizations of the input graph, and that its gradient flow converges exponentially to the unique minimizer in each homotopy class. The theoretical foundation builds on and refines classical results on discrete harmonic maps for surfaces of non-positive curvature by Colin de Verdière \cite{CdV1991}, adapting them to the setting of closed hyperbolic surfaces.

To allow our method to be applied to arbitrary input triangulated surface (not necessarily hyperbolic), we incorporate it into the framework of discrete conformal geometry. Using the discrete uniformization theorem \cite{Gu2018}, we first transform a given triangulated surface into a conformally equivalent hyperbolic surface equipped with a Delaunay triangulation. This provides a canonical geometric setting in which our structure-preserving discretization can be applied. The target hyperbolic surface is represented as a quotient of hyperbolic plane $\D/\Gamma$, where $\Gamma$ is a prescribed Fuchsian group. Taking the Euclidean harmonic map as the initializer, we implement a Riemannian gradient descent algorithm to find the unique hyperbolic harmonic map in the fixed homotopy class.

In summary, this paper introduces an algorithm for discrete harmonic maps that is both theoretically grounded and practically robust. By combining canonical edge weights, hyperbolic geometry, and discrete uniformization, we obtain a structure-preserving numerical method for mapping high-genus surfaces that respects their intrinsic geometry and yields high-quality correspondences. A direct application to remeshing is presented at the end of the article. 

\subsection{Related works}
 Many authors have designed various numerical algorithms for surface maps.  Relevant works include the Tutte embeddings/ Euclidean discrete harmonic maps ~\cite{tutte1963draw,gortler2006discrete,pinkall1993computing}, the elastic-type method~\cite{bajcsy1989multiresolution}, Mumford-Shah method~\cite{mumford1989optimal}, shape-preserving parameterization~\cite{floater1997parametrization}, discrete conformal maps/ circle packings ~\cite{angenent1999conformal,desbrun2002intrinsic,gu2003global,collins2003circle}, discrete Ricci flow~\cite{jin2008discrete}, optimal transport maps~\cite{haker2004optimal}, large diffeomorphic distance metric mapping (LDDMM)~\cite{beg2005computing,zhang2017frequency}, discrete Yamabe flow~\cite{lui2010detection}, and quasi-conformal maps~\cite{zeng2012computing,lui2013texture,zhu2022parallelizable}. 
 
 In addition to the applications mentioned in these papers, notable applications of surface maps also include the moving mesh method~\cite{li2001moving} in numerical PDE, implicitization of parametric surfaces~\cite{sederberg1995implicitization}, high-quality surface remeshing~\cite{eck1995multiresolution,remacle2010high}, surface matching~\cite{zhang1999harmonic}, statistical analysis of shapes~\cite{vaillant2004statistics}, 3D shape matching, recognition, and stiching~\cite{wang2007conformal}, and isogeometric analysis~\cite{xu2011parameterization}. 
 
 Despite the success of surface maps in many applications, few studies have focused on surfaces with genus $\geq 2$. Previous attempts to discrete harmonic maps for such surfaces make heuristic use of the Euclidean formula directly \cite{lui2014geometric,Gaster2018}, which fail to preserve harmonicity of isometries. Also, many existing algorithms for processing high-genus surfaces rely on partitioning them into several simply-connected pieces~\cite{isenburg2005streaming,marchandise2011high}. These facts motivate us to study a novel method for finding global maps between high-genus surfaces.

\section{Harmonic Maps and Their Discretization}

In this section, we discuss the theoretical background of classical harmonic maps for closed surfaces and their structure-preserving discretizations.
\subsection{Classical Harmonic Maps}  

Given two surfaces $S_1$ and $S_2$ in space with the same topology, finding a homeomorphism between them that minimizes distortion is a fundamental problem. The ideal case is an \textit{isometry}, which preserves distances between all pairs of points. However, isometries typically do not exist when the surfaces have different shapes.  

For simply-connected surfaces, a weaker alternative is a \textit{conformal map}, which preserves angles instead of distances. By the uniformization theorem, conformal maps always exist and are unique up to Möbius transformations. However, when the surfaces have non-trivial topology (i.e., genus \( g \geq 1 \)), a conformal map between them may not exist, as they might have different conformal structures.  

A more general approach is to consider \textit{harmonic maps}, which are maps $f:(S_1,h_1) \to (S_2,h_2)$ that minimize the Dirichlet energy within a given homotopy class:  

\begin{equation*}
E(f) = \frac{1}{2} \int_M \|\nabla_{h_2} f\|_{h_1}^2 \, dA_{h_1}.
\end{equation*}

Harmonic maps depend on the conformal structure of the input surface and the metric of the target surface. A classical result states that the harmonic map exists uniquely when the target surface has negative curvature everywhere and the harmonic map is a diffeomorphism if it is homotopic to a homeomorphism \cite{Jost1984}. When the surfaces are conformally equivalent via some conformal map, the harmonic map in that homotopy class coincides with the conformal map. If they are furthermore isometric, the harmonic map is itself an isometry, which is the ideal solution.  

Other approaches to comparing surfaces with different conformal structures include \textit{Teichmüller maps}, which minimize the maximal conformal distortion at every point, which have been implemented numerically \cite{ng2014teichmuller}. Comparing to Teichmüller maps, harmonic maps are more practical computationally for the variational setting, as harmonic maps minimize a $L^2$-functional instead of a $L^{\infty}$-norm.  

\subsection{Uniformization and Hyperbolic Plane}

Harmonic maps depend only on the conformal structure of the input surface, not on a particular choice of metric. Therefore, it is natural to consider canonical metrics associated with each conformal class. According to the uniformization theorem for Riemann surfaces, every Riemannian metric is conformally equivalent to a constant-curvature metric. By the Gauss--Bonnet theorem, the sign of the curvature $K$ is determined by the Euler characteristic of the surface, resulting in three possible types of canonical metrics: spherical ($K = 1$), Euclidean ($K = 0$), and hyperbolic ($K = -1$).

In particular, the canonical metric associated with a torus ($g = 1$) within its conformal class is Euclidean ($K = 0$). For surfaces of genus $g \geq 2$, the canonical metric is hyperbolic ($K = -1$). Such constant-curvature metrics are uniquely determined by the conformal structure of the Riemann surface. 

This paper focuses on computing the harmonic maps between closed surfaces of genus $g \geq 2$ equipped with a hyperbolic metric. Such a surface can be represented as a quotient of the hyperbolic plane by a subgroup of hyperbolic isometries. Here the hyperbolic plane is characterized as the simply-connected Riemann surface equipped with a Riemannian metric of constant curvature $-1$. Below we review basic properties of the hyperbolic plane, which are key to our theoretical reasoning and the design of the numerical algorithm. 

Several equivalent models are commonly used to define the hyperbolic plane and we shall focus on the Poincaré disk model. The Poincaré disk $\D:=\{z:|z|<1\}$ is a simply-connected surface equipped with Riemannian metric 
\begin{equation*}
    \rho_{\D} := \dfrac{2|dz|}{1-|z|^2}.
\end{equation*}
Orientation-preserving isometries of the Poincaré disk model of the hyperbolic plane are exactly the Möbius transformations that map the unit disk to itself. Because of the non-Euclidean metric, geodesics of the hyperbolic plane in the disk model are either Euclidean diameters of the unit circle or Euclidean circles orthogonal to the unit circle. 

Given two points $p$ and $q$ in the hyperbolic plane, there is a unique geodesic connecting them. The length of the geodesic segment between $p$ and $q$ is called the geodesic distance between $p$ and $q$. For example, in the Poincaré disk, the distance between arbitrary $p$ and $q$ is given by
\begin{equation}
d(p,q) = \operatorname{arcosh} (1+\delta(p,q)),
    \label{eq:hyperbolic_distance}
\end{equation}
where
\begin{equation}
    \delta(p,q) = 2\dfrac{||p-q||^2}{(1-||p||^2)(1-||q||^2)}.
\end{equation}
Therefore, a hyperbolic polygon in the Poincaré disk model is bounded by a finite number of geodesic segments, rather than Euclidean straight lines. Similar to the Euclidean geometry, we have a hyperbolic law of cosines relating the hyperbolic distances and angles of a hyperbolic triangle. One of the equation is given by
\begin{equation*}
    \cosh a = \cosh b\cosh c - \sinh b\sinh c\cos A.
\end{equation*}
where $a$, $b$, and $c$ are lengths of the edges respectively and $A$ is the angle opposite to $a$. This law is crucial when we need to construct a hyperbolic polygon, which is key for the construction of a hyperbolic surface as we will see in the algorithm part.

\subsection{Discrete Harmonic Maps}  

In the discrete setting, we begin with a triangulated hyperbolic surface $S_1$ and write $f^{\dagger}: (V,E) \to S_1$ the embedding of the 1-skeleton graph. The target surface is another hyperbolic surface $S_2$, which is identified with $S_1$ via a homeomorphism $h:S_1 \to S_2$. The composition $h\circ f^{\dagger}: (V,E) \to S_2$ yields a topological triangulation of $S_2$. Depending on the choice of $h$, the image of edges are generally not geodesics and the underlying surface map has huge distortion. Our goal is to find an optimal embedding $f: (V,E) \to S_2$ of the 1-skeleton graph among all geodesic realizations within the homotopy class of $h\circ f^{\dagger}$.

Every geodesic realization of the graph to the hyperbolic surface $S_2$ is associated with the discrete Dirichlet energy
\begin{equation}
\label{eq:discrete_dirichlet_energy}
	D_c(f) = \frac{1}{2}\sum_{ij} c_{ij} \ell_{ij}^2,
\end{equation}
where \( c:E \to \mathbb{R}_{>0} \) with \( c_{ij}=c_{ji} \) is a prescribed system of edge weights, and \( \ell_{ij} \) is the hyperbolic length of the geodesic edge joining vertices \( i \) and \( j \) measured on $S_2$. A realization \( f \) is called a \emph{discrete harmonic map} if it is a critical point of the energy \( D_c \). Equivalently, it satisfies for every vertex \( i \in V \),
\begin{equation}\label{eq:discreteharmonic}
\sum_{j} c_{ij} \ell_{ij} U_{ij} = 0,
\end{equation}
where \( U_{ij} \) is the outward unit tangent vector to the geodesic from vertex \( i \) to vertex \( j \).

An example is the \textit{Tutte embedding}~\cite{tutte1963draw}, where a planar graph is embedded in the Euclidean plane such that each interior vertex is the weighted average of its neighbors. This concept extends naturally to surfaces with non-trivial topology \cite{CdV1991,Kotani2001,Toru2021,Gaster2018}. Colin de Verdière established the following foundational result.

\begin{theorem}[Colin de Verdière \cite{CdV1991}]
    Given a topological triangulation on a closed surface of non-positive curvature together with a choice of positive edge weights \( c \), there exists a unique discrete harmonic map \( f \) in the homotopy class of the triangulation. This map is the unique minimizer of the energy \( D_c \) in the given homotopy class. Furthermore, the discrete harmonic map is an embedding.
    \label{thm:colin_de_verdiere}
\end{theorem}

To make the exposition self-contained and tailored to our setting, we reformulate the main ideas of the proof in the specific context of hyperbolic surfaces. \smallskip

\noindent\textbf{Existence.} A geodesic realization $f:(V,E) \to S$ is determined by vertex positions and the homotopy class. Fixing a spanning tree in the graph, each geodesic realization leads to a lift of the spanning tree to the universal cover $\mathbb{D}$ and can be regarded as an element in the product space $\mathbb{D}^{|V|}:=\mathbb{D} \times \mathbb{D} \dots \times \mathbb{D}$, where each component corresponds to the position of a vertex. Conversely, a geodesic realization can be recovered from an element in $\mathbb{D}^{|V|}$ via projecting the corresponding spanning tree on $\mathbb{D}$ to the surface $S$ and recovering the remaining edges by the homotopy data.

Two elements in $\mathbb{D}^{|V|}$ project to the same geodesic realization on $S$ if they differ by a hyperbolic isometry as a deck transformation and so we can always normalize the first vertex to lie in a compact fundamental domain $\mathcal{F} \subset \mathbb{D}$. It shows that the space of geodesic realizations within a fixed homotopy class is modeled by a manifold $\mathcal{F} \times \mathbb{D} \dots \times \mathbb{D}$ of dimension $2|V|$.

One can check that the Dirichlet energy is a proper functional on the space of geodesic realizations. Any two vertices can be connected via at most $|E|$ edges. Fixing any positive number $R$, we denote $B_R \subset \mathbb{D}$ the compact set 
\[
B_R:=\{ x \in \mathbb{D}| d_{\mathbb{D}}(x,y) \leq |E| \cdot \sqrt{\frac{R}{\min_{e \in E} c_e}} \text{ for some } y \in \mathcal{F}\}
\]
and observe that the closed subset $D_c^{-1}([-R, R]) \subset \mathcal{F} \times B_R \dots \times B_R$ is compact. Indeed, any geodesic realization $f$ lying out of $\mathcal{F} \times B_R \dots \times B_R$ indicates the existence of a geodesic edge $e$ with length larger than $\sqrt{\frac{ R}{\min_{e \in E} c_e}}$ and hence the energy
\[
D_c(f) > c_e \cdot (\sqrt{\frac{R}{\min_{e \in E} c_e}})^2 = R.
\]
Because $D_c^{-1}([-R, R])$ is compact for any $R>0$, we conclude that the Dirichlet energy is a proper functional on the space of geodesic realizations.

Since the energy is always positive, its infimum is finite. There exists a sequence of geodesic realization with decreasing energy that approaches the infimum. Because $D_c^{-1}([0, R])$ is compact for any $R>0$, the sequence converges to a geodesic realization that minimizes the energy and hence is a discrete harmonic map. \smallskip

\noindent\textbf{Uniqueness.} Within the fixed homotopy class, the discrete harmonic map exists uniquely. One observes that given a geodesic realization $f$, we can parametrize each geodesic edge $ij$ via $f_{ij}(s)$ with constant speed, where $s\in [0,1]$. In this way, the Dirichlet energy can be rewritten as
\[
D_c(f_t)=\sum_{ij} c_{ij} \int_0^1 || \frac{\partial f_{t,ij}}{\partial s}||^2 ds.
\]
Suppose there are two discrete harmonic maps \( f_0 \) and \( f_1 \) in the same homotopy class. They can be connected by a family of realizations $f_t$ for $t \in [0,1]$ such that for every edge $ij$ and $s\in [0,1]$, $f_{t,ij}(s)$ is a constant-speed geodesic joining $f_{0,ij}(s)$ and $f_{1,ij}(s)$ as $t\in [0,1]$ varies. Writing the vector field $W:= \frac{\partial f_t}{\partial t}$, one has the second derivative (See \cite[Lemma A.1]{Toru2021})
\[
\frac{d^2}{dt^2} D_c(f_t) =\sum_{ij} c_{ij} \int_0^1 (|| \nabla_{\frac{\partial}{\partial s}} W_{ij}||^2 + ||W^{\perp}_{ij}||^2 || \frac{\partial f_{t,ij}}{\partial s}||^2) ds \geq 0.
\]
where $\nabla$ is the Levi-Civita connection and $W^{\perp}$ is the component of $W$ that is perpendicular to the geodesic. Since $\frac{d}{dt} D_c(f_t)|_{t=0}=0=\frac{d}{dt} D_c(f_t)|_{t=1}$, one can deduce that the vector field $W$ is trivial for all $t$ and hence $f_0=f_1$. \smallskip

\noindent\textbf{Embeddedness.} The Gauss–Bonnet theorem implies that adjacent faces remain unfolded under a discrete harmonic map. Suppose no edges degenerate. For a geodesic triangle $ijk$, we denote $\alpha^{i}_{jk},\alpha^{j}_{ki},\alpha^{k}_{ij} \in [0,\pi)$ the corner angles of the triangle. Equation \eqref{eq:discreteharmonic} yields that at every vertex $i$, the outward tangent vectors to the geodesics cannot be contained in any half space of the tangent plane and hence
\begin{align}\label{eq:vertex}
    \sum_{jk} \alpha^{i}_{jk} \geq 2 \pi
\end{align}
where the sum is over all faces $ijk$ that contains vertex $i$. We denote $\area(ijk)$ the signed area of the hyperbolic triangle $ijk$ under the discrete harmonic map. The Gauss-Bonnet theorem yields
\begin{equation}\label{eq:face}
    \pi - (\alpha^{i}_{jk}+\alpha^{j}_{ki}+\alpha^{k}_{ij})=| \area(ijk)| \geq \area(ijk)
\end{equation}
where the signed area is negative if the orientation of $ijk$ induced from $S_1$ is reversed under the discrete harmonic map to $S_2$. Summing Equation \eqref{eq:face} over all faces yields
\begin{align*}
    \sum_{ijk}(\pi - (\alpha^{i}_{jk}+\alpha^{j}_{ki}+\alpha^{k}_{ij})) \geq \sum_{ijk} \area(ijk) = -2\pi \chi
\end{align*}
where $\chi=|V|-|E|+|F|$ is the Euler characteristic. Together with equation \eqref{eq:vertex}, it implies
\begin{align*}
    2\pi \chi \geq -\pi |F| + \sum_i \sum_{jk} \alpha^{i}_{jk} \geq -\pi |F| + 2\pi |V| \geq 2\pi \chi
\end{align*}
since $2|E|=3|F|$ holds for a triangulation. Thus, all the equalities hold and adjacent faces remain unfolded under a discrete harmonic map. \smallskip

\noindent\textbf{Gradient flow.} Since the energy function $D_c$ is non-negative, proper, and continuously differentiable on a finite-dimensional manifold, its negative gradient flow exists for all time and converges to the unique discrete harmonic map in a homotopy class. Moreover, the formula for the second derivative of $D_c$ shows that the Hessian at the discrete harmonic map is positive definite~\cite[Theorem 2.1]{Toru2021}. A consequence of this fact is that as long as $D_c$ is bounded through the flow line, the negative gradient flow of $D_c$ will converge to a critical point and the convergence is exponentially fast as $t\to\infty$. For details and proofs of this result, the interested readers are suggested to refer to the book~\cite{jost2008riemannian}. \medskip

Analogous results on Euclidean tori (genus $g=1$) hold as well, but with different arguments (See \cite{gortler2006discrete,Lin2020}).

\subsubsection{Geometric Edge Weights}  

The definition of a discrete harmonic map depends critically on the choice of edge weights $c:E \to \mathbb{R}$. Different choices of weights can lead to different homeomorphisms between the surfaces. A good choice of edge weights should reflect the geometry of the input surface $S_1$ and yield a structure-preserving discretization, in the sense that it satisfies the following criteria:
\begin{enumerate}
         \item \textbf{Injectivity}: All edge weights must be positive to ensure that the resulting discrete harmonic map is injective.  

      \item \textbf{Harmonicity of isometries}: 
      If $S_2$ is isometric with $S_1$ via an isometry $h:S_1 \to S_2$, then the composition $h \circ f^{\dagger}$ is discrete harmonic map.

\end{enumerate}

The uniformization theorem allows us to focus on geometric edge weights derived from constant-curvature metrics and the criteria naturally lead to a unique, canonical choice of such edge weights \cite{Lam2022,Lam2024harmonic}. In the hyperbolic setting, these edge weights arise as a natural generalization of the classical \emph{cotangent weights} in the Euclidean case. The proofs involved characterizing the minimizer of Dirichlet energy of discrete harmonic maps over the space of marked hyperbolic metrics. 

\medskip

\noindent \textbf{Euclidean metrics ($K = 0$).}  
For tori $(g=1)$, the uniformization theorem ensures that the conformal structure is represented by a Euclidean metric. It is shown in \cite{Lam2022} that the criterion on harmonicity of isometries selects the standard choice of edge weights, known as the \emph{cotangent weights} $c:E \to \mathbb{R}$ defined by (See Figure \ref{fig:orientation} for notations)
\begin{align}
\begin{split}
c_{ij} &:= \frac{1}{2} \left( \tan\left( \frac{\alpha^{i}_{jk} + \alpha^{j}_{ki} - \alpha^{k}_{ij}}{2} \right) 
+ \tan\left( \frac{\alpha^{i}_{lj} + \alpha^{j}_{il} - \alpha^{l}_{ji}}{2} \right) \right) \\
&= \frac{1}{2} \left( \cot \alpha^{k}_{ij} + \cot \alpha^{l}_{ji} \right).
\end{split}
\end{align}
These weights arise naturally from the finite element discretization of the Dirichlet energy \cite{pinkall1993computing} and discrete harmonic functions in electric networks \cite{Duffin1959}.

\medskip

\noindent \textbf{Hyperbolic metrics ($K = -1$).}  
For closed surfaces of genus $g \geq 1$, the uniformization theorem guarantees a unique hyperbolic metric in each conformal class. In this case, the criterion on harmonicity of isometries determines a canonical choice of edge weights \cite{Lam2024Lap,Lam2024harmonic}:
\begin{align}
c_{ij} := \left( \tan\left( \frac{\alpha^{i}_{jk} + \alpha^{j}_{ki} - \alpha^{k}_{ij}}{2} \right) 
+ \tan\left( \frac{\alpha^{i}_{lj} + \alpha^{j}_{il} - \alpha^{l}_{ji}}{2} \right) \right)
\cdot \frac{\tanh\left( \frac{\ell_{ij}}{2} \right)}{\ell_{ij}}.
\label{eq:cij}
\end{align}
This formula generalizes the Euclidean cotangent weights to the hyperbolic setting.

\begin{figure}
    \centering
	\includegraphics[width=0.55\textwidth]{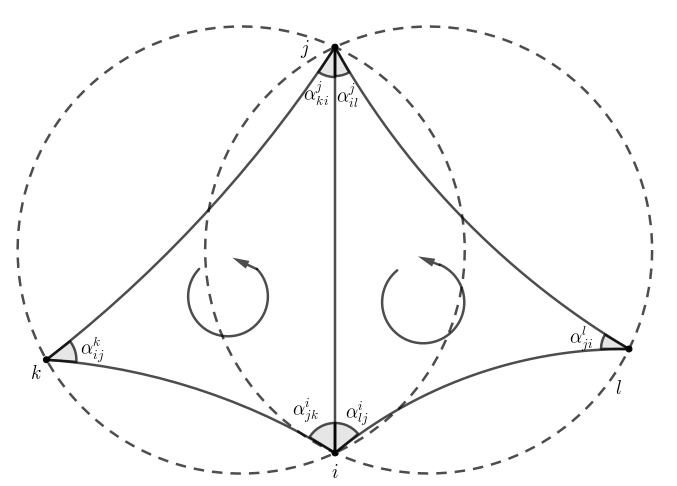}
    \caption{Two neighboring hyperbolic triangles share a common edge $ij$.}
    \label{fig:orientation}
\end{figure}

\subsubsection{Injectivity -- Delaunay Condition}  
The positivity of the canonical edge weights is equivalent to the condition that the triangulation of the input surface is \emph{Delaunay}.

A geodesic triangulation is Delaunay if it satisfies the empty circle condition: the circumcircle of each face contains no vertex of an adjacent face in its interior. This condition is equivalent to requiring that, for every edge shared by two triangles $ijk$ and $jil$, the following inequality holds:
\[
\frac{\alpha^{i}_{jk} + \alpha^{j}_{ki} - \alpha^{k}_{ij}}{2} + \frac{\alpha^{i}_{lj} + \alpha^{j}_{il} - \alpha^{l}_{ji}}{2} \geq 0
\]
On the other hand, the sum of three angles in a triangle is equal to $\pi$ (Euclidean) or strictly less than $\pi$ (hyperbolic). Thus the Delaunay condition is equivalent to the positivity of edge weights.

\subsubsection{Harmonicity of isometries}

A key feature that distinguishes our approach from previous work is the use of canonical edge weights, which ensure that isometries remain harmonic in the discrete theory, reflecting their classical behavior.

\begin{theorem}[\cite{Lam2024harmonic}]
    Let $S_1$ by a hyperbolic surface equipped with a geodesic triangulation $f^{\dagger}:(V,E)\to S_1$ and $c:E \to \mathbb{R}$ be the canonical edge weight. Suppose the target surface $S_2$ has the same hyperbolic structure and identified via an isometry $h:S_1 \to S_2$. Then $h\circ f^{\dagger}$ is a discrete harmonic map to $S_2$ with respect to the canonical edge weight. 
\end{theorem}

\subsection{Discrete Uniformization}

Given a triangulated surface $(S,\ell)$ in space with prescribed edge lengths, a remaining problem is to compute a hyperbolic surface in the same conformal class, equipped with a Delaunay geodesic triangulation. We adopt the framework of discrete conformal geometry to address this challenge.

In discrete conformal geometry \cite{Luo2004,Bobenko2015,Gu2018}, two piecewise Euclidean metrics $(S,\ell)$ and $(S,\tilde{\ell})$ with the same number of vertices are said to be discretely conformally equivalent if they are related by a sequence of vertex scalings and edge flips that preserve Delaunay triangulations. Specifically, there exists a sequence of triangulations $T_0, T_1, \dots, T_k$, where each $T_{i+1}$ is obtained from $T_i$ by an edge flip, such that within each fixed triangulation, the edge lengths are related by a vertex-based scaling $u:V \to \mathbb{R}$ satisfying
\[
\tilde{\ell}_{ij} = e^{\frac{1}{2}(u_i + u_j)} \ell_{ij}.
\]
This vertex scaling mimics the conformal equivalence of Riemannian metrics in the smooth setting.

An alternative, more geometric, viewpoint is to associate each piecewise linear metric with a complete hyperbolic metric with cusps. In this perspective, two piecewise linear metrics are discretely conformally equivalent if they correspond to the same complete hyperbolic metric with cusps \cite{Bobenko2015}. This interpretation allows the notion of discrete conformal equivalence to be extended naturally to piecewise hyperbolic metrics.

The discrete uniformization problem is then to find, within the discrete conformal class of a given triangulated surface, a metric of constant negative curvature together with a Delaunay triangulation.

\begin{theorem}[Gu--Guo--Luo--Sun--Wu {\cite{Gu2018}}]
Let $(S,\ell)$ be a closed triangulated surface of genus $g \geq 2$, equipped with a piecewise Euclidean metric. Then there exists a unique hyperbolic surface equipped with a Delaunay geodesic triangulation that is discretely conformally equivalent to $(S,\ell)$.
\end{theorem}

Via discrete uniformization, we obtain a canonical Delaunay triangulation on a hyperbolic surface. This provides a natural setting in which we could compute discrete harmonic maps from it to other surfaces.

\section{Algorithm}

The input data of our algorithm is a compact and connected triangulated surface $(V,E,F)$ without boundary, whose genus is at least $2$. In addition, this surface is equipped with a flat discrete hyperbolic metric $\ell:E\to\R_{>0}$. To obtain a flat hyperbolic metric, it is quite safe to apply the discrete hyperbolic Ricci flow algorithm~\cite{jin2008discrete}, whose convergence is proved in~\cite{chow2003combinatorial}. With the metric $\ell$, we can compute the hyperbolic edge weight $c_{ij}$ through the formula~\eqref{eq:cij}. In case the initial surface is not a Delaunay triangulation, the weight of some edges may be negative. To ensure that the final map is an embedding, we will reset the weight of these edges to a prescribed positive number. The output of this algorithm is expected to be a discrete harmonic map of $f:(V,E)\to S$ onto a target hyperbolic surface $S$, which minimizes the Dirichlet energy~\eqref{eq:discrete_dirichlet_energy}. 

\subsection{Representation of the target surface}
To construct a target surface, we recall that by the uniformization theorem, a closed hyperbolic surface is the quotient space $\mathbb{D}/\Gamma$ for some faithful representation of the fundamental group to a Fuchsian group $\Gamma$, which is a subgroup of hyperbolic isometries. It is not necessary for us to construct a Fuchsian group directly. Instead, we could rely on a famous theorem by Poincaré, which states that if $\bar{R}$ is a finite-sided closed convex polygon whose sides are identified in pairs by isometries $G_0=\{g_1,\cdots,g_r\}$ of the Poincaré disk, and suppose the cycle condition holds, then $G_0$ generates a discrete subgroup $\Gamma$ of the isometry group of $\mathbb{D}$, and $R$ is a fundamental domain under the action of $\Gamma$. Here, a convex polygon should be understood in the sense of hyperbolic geometry, which means that its edges are geodesics in the Poincaré disk $\mathbb{D}$ and it is a geodesically convex domain in $\mathbb{D}$. To explicitly get a paired hyperbolic polygon, we could rather manually construct one or rely on some existing methods such as the hyperbolic Ricci flow algorithm~\cite{jin2008discrete}, which can output the canonical polygon of a hyperbolic surface with respect to some flat hyperbolic metric. 

\begin{figure}[t]
\centering
  \includegraphics[width=0.99\linewidth]{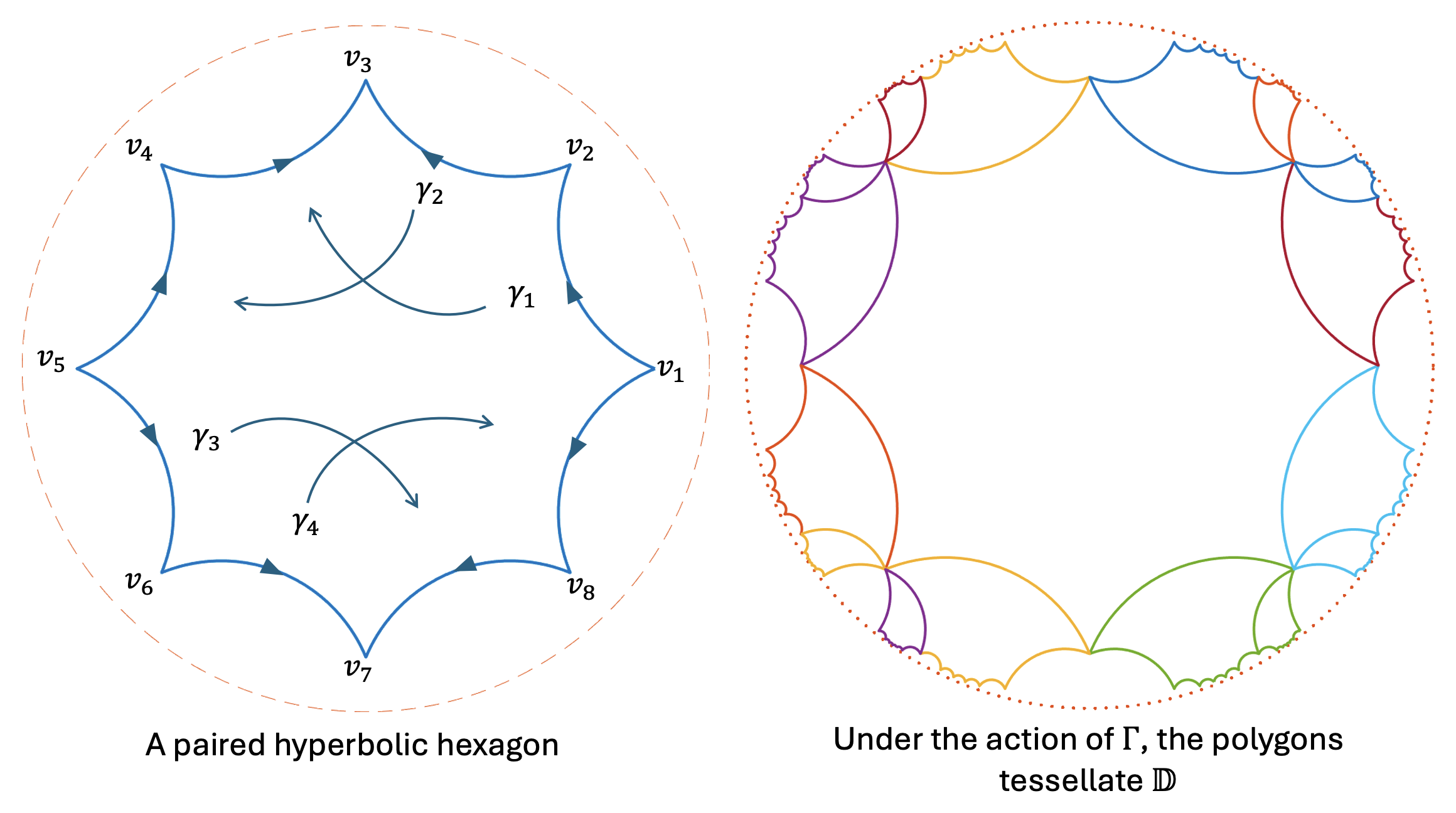}
\caption{Construction of a hyperbolic hexagon with paired edges.}
\label{fig:hyperbolic_octagon}
\end{figure} 

Here, we give an example of a genus-$2$ surface constructed in this way, which we will use as a target surface in the numerical experiments. Let $\bar{R}$ be the closed hyperbolic octagon with vertices $v_k = se^{i(k-1)\pi/4},k=1,\cdots,8$, where $s>0$ is uniquely determined so that the sum of the internal angles of $\bar{R}$ is $2\pi$. As shown in the figure~\ref{fig:hyperbolic_octagon}, the vertices of $R$ are named $v_1,\cdots,v_8$, in counterclockwise order. To make $R$ topologically a genus-$2$ surface, the sides $r_1,r_2,r_5$ and $r_6$ are paired with $r_3^{-1},r_4^{-1},r_7^{-1}$ and $r_8^{-1}$ respectively. The side-pairing transformation between $r_1$ and $r_3^{-1}$ is the unique element $\gamma_1\in \mbox{Aut}(\D)$ satisfying
\begin{equation*}
    \gamma_1(v_1) = v_4\ \ \text{and}\ \ \gamma_1(v_2) = v_3.
\end{equation*}
The other side-pairing transformations $\gamma_2,\gamma_3,\gamma_4\in\mbox{Aut}(\D)$ are given in a similar way. These transformations and their inverses generate a discrete subgroup $\Gamma$ of $\mbox{Aut}(\D)$. It is easy to see that one representative of $\D/\Gamma$ is just $\bar{R}$ with half the boundary points removed. Heuristically, this means that the polygon $R$ tessellates $\D$ under the action of $\Gamma$. The vertices $\{v_k:1\leq k\leq8\}$ can be paired by an elliptic cycle
\begin{equation*}
    v_1 \xrightarrow{\gamma_1} v_4 \xrightarrow{\gamma_2^{-1}} v_3 \xrightarrow{\gamma_1^{-1}} v_2 \xrightarrow{\gamma_2} v_5 \xrightarrow{\gamma_3} v_8 \xrightarrow{\gamma_4^{-1}} v_7 \xrightarrow{\gamma_3^{-1}} v_6 \xrightarrow{\gamma_4} v_1.
    \label{eq:elliptic_cycle}
\end{equation*}

Besides, it is well known that the fundamental group of a hyperbolic surface, the group of deck transformations of its universal cover, and the Fuchsian group that generates it are isomorphic to each other. So, for the sake of simplicity, we will commonly denote all of the three groups by $\Gamma$ in the remaining article.

\subsection{A Riemannian Gradient Descent Algorithm}
Now, suppose that we have fixed the target surface $S$ as an appropriate hyperbolic polygon $\bar{R}$ with sides $\{r_1,\cdots,r_{4g}\}$ and a Fuchsian group $\Gamma$ generated by the prescribed side-pairings, where $g$ is the genus of the initial surface. We immediately notice that it is impossible to embed $(V,E,F)$ directly into $\D$ because $\D$ is simply-connected while $(V,E,F)$ is not. Our resolution is to lift a map $f:(V,E,F)\to\D/\Gamma$ to the developing map from its universal cover to $\D$. As a result, the first step of our algorithm is to find a homology basis of $(V,E,F)$ and cut it along this basis to produce a simply-connected triangulation $(\tilde{V},\tilde{E},F)$. $(V,E)$ and $(\tilde{V},\tilde{E})$ are related by the canonical projections $\pi_V:\tilde{V}\to V$ and $\pi_E:\tilde{E}\to E$. For better demonstration of our algorithm, we require that the homology basis be based at one given point, and that any two different loops can only cross at the base point. In terms of the developing map $\tilde{f}:(\tilde{V},\tilde{E})\to\D$, we aim to minimize the energy
\begin{equation}
    D_c(\tilde{f}) = \frac{1}{2}\sum_{ij\in E} c_{ij} \ell_{ij}^2
\end{equation}
over all geodesic realizations $\tilde{f}$ of $(\tilde{V},\tilde{E})$ into $\D$, which satisfy the hard constraint that for any $i\in\tilde{V}$ and $\gamma\in\pi_1(S)$, we have
\begin{equation}
    \tilde{f}(\gamma(i)) = \gamma(\tilde{f}(i))\ \ \text{for some } \gamma \in \Gamma.
    \label{eq:hard_constraint}
\end{equation}
For any $\tilde{f}$ satisfying the constraint, the image of $\tilde{f}$ tessellates $\D$ under the action of $\Gamma$, hence giving the developing map from the universal cover $(V,E)$ to $\D$. The initial map $\tilde{f}_0$ can be found by an Euclidean approach. Taking the preimages of the base point as division points, we can separate the boundary of $(\tilde{V},\tilde{E})$ into almost disjoint segments $s_1,s_2,\cdots,s_{4g}$. By some suitable re-ordering, we could assume that $s_i$ is paired with $s_j$ if $r_i$ is paired with $r_j$. By prescribing a boundary condition $\restr{\tilde{f}}{\partial (\tilde{V},\tilde{E})}$ satisfying $\tilde{f}(s_i)=r_i$ for each $i$ and
\begin{equation}
    \tilde{f}(s_i) = \gamma\circ\tilde{f}(s_j)
\end{equation}
if $r_i$ and $r_j$ is paired by $\gamma\in \Gamma$, we can compute a Euclidean harmonic map from $(\tilde{V},\tilde{E})$ to $R$, with respect to some Euclidean edge weights. With the initial map $\tilde{f}_0$, our goal is to find a one-parameter family of developing maps $\tilde{f}_t$ so that $\tilde{f}_t$ satisfy the hard constraint~\eqref{eq:hard_constraint} for each time $t$, and $\tilde{f}_t$ converges to the desired hyperbolic harmonic map as $t\to\infty$. One potential approach is to minimize $D_c(\tilde{f})$ by gradient descent. Let $u_i = (x_i,y_i)$ and $u_j=(x_j,y_j)$ be two points in $\D$. We denote by $\ell_{ij}$ the length of the hyperbolic geodesic connecting $u_i$ and $u_j$. Using the formula for hyperbolic distance~\eqref{eq:hyperbolic_distance}, we take the usual partial derivative in $\mathbb{R}^2$, and get
\begin{equation}
    \dfrac{\partial \ell_{ij}}{\partial x_i} = \dfrac{4}{\sqrt{(1+\delta(u_i,u_j))^2-1}}\dfrac{\big((x_i-x_j)(1-\norm{u_i}^2)+x_i\norm{u_i-u_j}^2\big)(1-\norm{u_j}^2)}{(1-\norm{u_i}^2)^2(1-\norm{u_j}^2)^2}
\end{equation}
where
\begin{equation*}
    \delta(u_i,u_j) = 2\,\dfrac{\norm{u_i-u_j}^2}{(1-\norm{u_i}^2)(1-\norm{u_j}^2)}.
\end{equation*}
The formula giving partial derivatives with respect to $x_j$, $y_i$, and $y_j$ are similar. By observing the formula, we instantly notice a significant problem. That is, when $\norm{u_i}$ approaches $1$, the denominator of $\partial \ell_{ij}/\partial x_i$ will become very large. As a consequence, when the step size is not small enough, a point near $\partial\mathbb{D}$ is very likely to go out of $\mathbb{D}$ after a step of gradient descent, which makes the algorithm very unstable. On the other hand, we recall a basic fact that a hyperbolic geodesic with any initial point in $\mathbb{D}$ and any initial direction is well-defined for any $t>0$. This motivates us to design an optimization algorithm under the Riemannian setting. From now on, we suppose that all the differential operators we use are compatible with the Riemannian structure $(\D,g)$. Again, we let $\ell_{ij}$ denote the geodesic distance between $u_i$ and $u_j$. It can be easily seen that
\begin{equation}
    \frac{1}{2}\dfrac{\partial \ell^2_{ij}}{\partial u_i} = \ell_{ij}U_{ij}
\end{equation}
where $U_{ij}\in T_{u_i}\D$ is the unit tangent vector at $u_i$ of the geodesic traveling from $u_i$ to $u_j$. We must be more careful when taking partial derivatives of the energy $D_c(\tilde{f})$. The reason is that the summation is taken over all edges $ij\in E$ while $\ell_{ij}$ in the summand is the geodesic distance between $\tilde{f}(i)$ and $\tilde{f}(j)$ as for the developing map $\tilde{f}:\tilde{V}\to\D$. Thus, we need to discuss the following $3$ cases. First, we suppose that a vertex $i$ is an interior vertex of $(\tilde{V},\tilde{E})$. This case is trivial because the edges $ij\in\tilde{E}$ incident with $i\in\tilde{V}$ are exactly the edges $ij\in E$ incident with $i\in V$. As a result, we get
\begin{equation}
    \dfrac{\partial D_c(\tilde{f})}{\partial\tilde{f}(i)} = \sum_{j\sim i,j\in\tilde{V}} c_{ij}\ell_{ij}U_{ij}.
\end{equation}
For the other two cases, we assume that $i$ is a boundary vertex of $(\tilde{V},\tilde{E})$. Under this assumption, the edges $ij\in\tilde{E}$ incident with $i\in\tilde{V}$ do not match the edges incident with $\pi_V(i)$ in $(V,E)$. The basic idea to tackle this problem is as follows. Let $\pi$ denote the covering map from $\D$ to $(V,E,F)$, which is induced by $\tilde{f}$. Also, let $U_i$ denote the set consisting of all edges incident with $\pi_V(i)$ in $(V,E)$. Then, there exists a discrete space $D_i$ such that $\pi^{-1}(U_i)=\sqcup_{d\in D_i}W_d$. Among them, there is a unique one containing $\tilde{f}(i)$ and we denote it by $\tilde{W}$. The partial derivative with respect to $\tilde{f}(i)$ is then given by
\begin{equation}
    \dfrac{\partial D_c(\tilde{f})}{\partial\tilde{f}(i)} = \sum_{ij\in\tilde{W}} c_{ij}\ell_{ij}U_{ij} \ \ \text{if }i\in\partial\,(\tilde{V},\tilde{E},F).
    \label{eq:gradient_vector}
\end{equation}
Algorithmically, we divide this situation into two cases, depending on whether $\pi_V(i)$ is the base point of the homology basis along which we cut $(V,E)$. If $\pi_V(i)$ is not the base point, we further suppose that $i\in s_{k_1}$, which is paired with $s_{k_2}$ through $s_{k_2} = \gamma(s_{k_1})$ for some $\gamma\in \Gamma$. It is easy to see that for each interior vertex $j\in\tilde{V}$ in the neighborhood of $\gamma(\tilde{f}(i))$, $\gamma^{-1}(\tilde{f}(j))$ belongs to the neighborhood of $\tilde{f}(i)$ in the universal cover. The edges formed in this way, together with the edges connected to $i$ in $(\tilde{V},\tilde{E})$, correspond to the set $\tilde{W}$ mentioned above. On the other hand, if $\pi_V(i)$ is a base point, we can find a representative in $\tilde{E}$ for each edge incident with $\pi_V(i)$ in $(V,E)$. Each representative is incident with a point $j\in\tilde{V}$ such that $\pi_V(j)$ is the base point. Recall that all the points $j\in\tilde{V}$ corresponding to the base point can be mapped to $i$ through an elliptic cycle. An example of the elliptic cycle of a genus-$2$ surface is given in the previous text~\eqref{eq:elliptic_cycle}. Using the elliptic cycle, each representative can be mapped to an edge in $\tilde{W}$. Together, they give the partial derivative of $D_c(\tilde{f})$ with respect to $\tilde{f}(i)$. 

\begin{figure}[t]
\centering
  \includegraphics[width=0.99\linewidth]{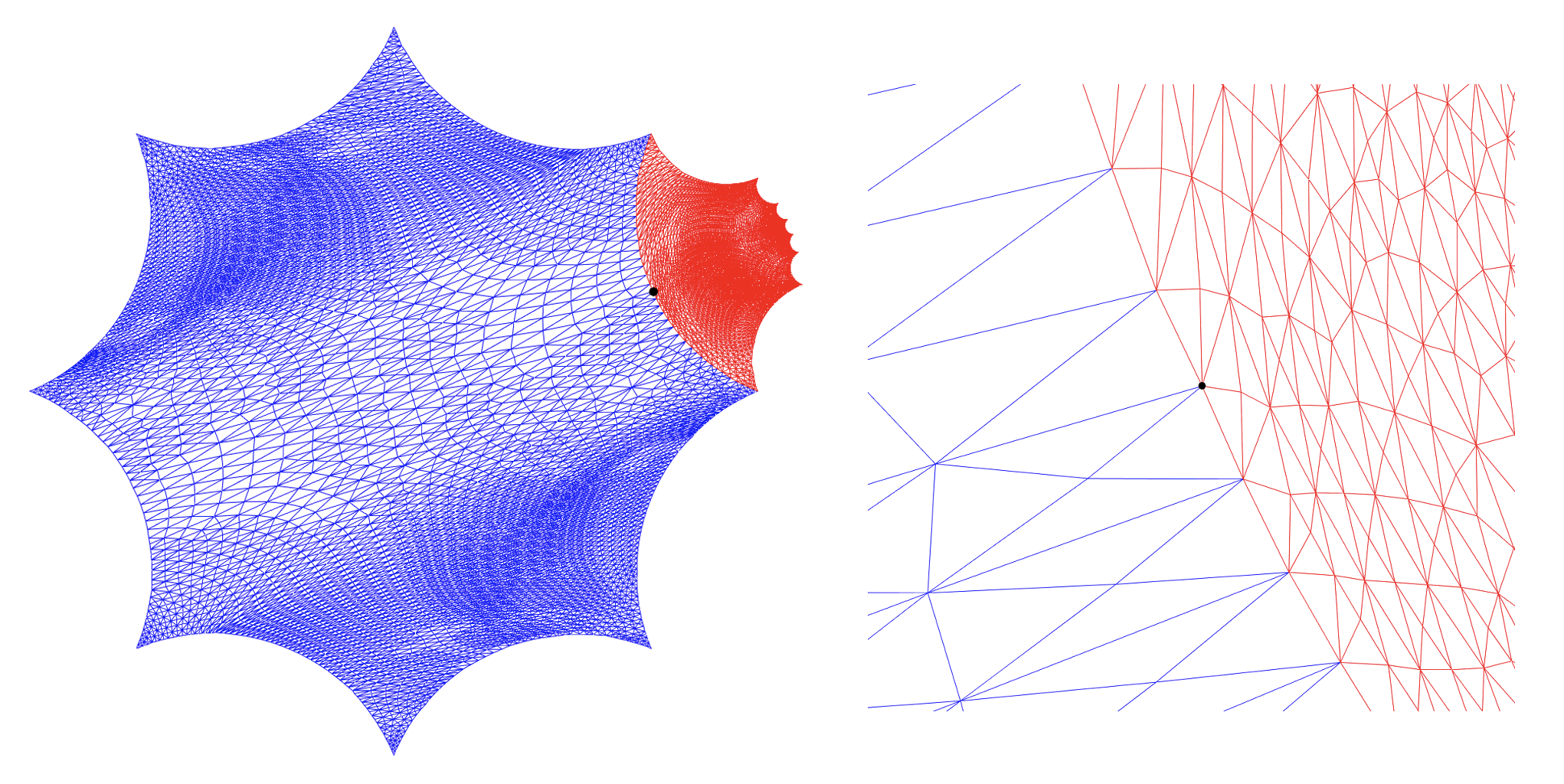}
\caption{The figures show how we compute the descent direction for boundary vertices. A chosen boundary vertex is pinned in black. We count all edges in the universal cover.}
\label{fig:descent_boundary}
\end{figure} 

\begin{figure}[t]
\centering
  \includegraphics[width=0.99\linewidth]{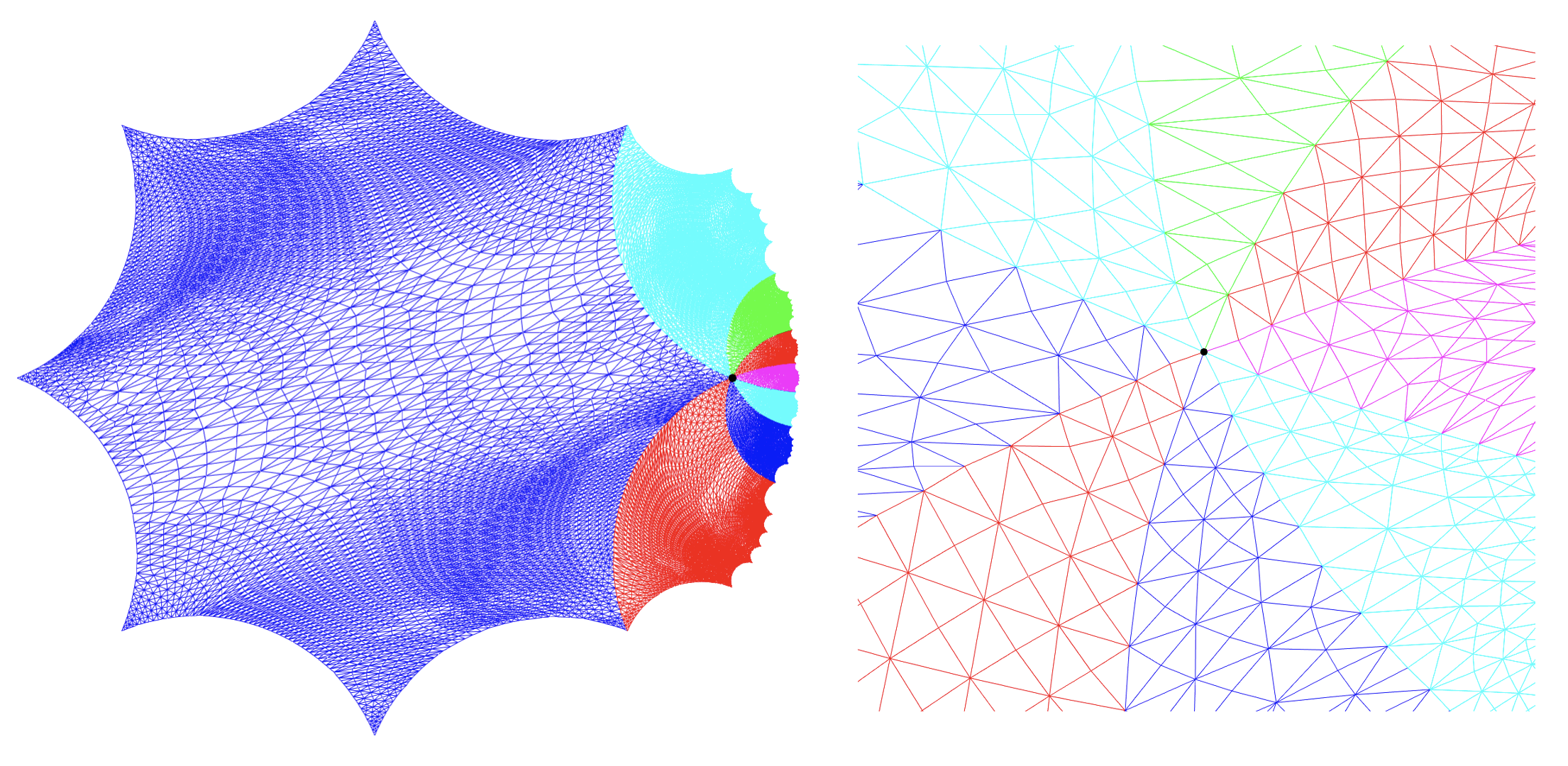}
\caption{The figures show how we compute the descent direction for corner vertices. A chosen corner vertex is pinned in black. We count all edges in the universal cover.}
\label{fig:descent_corner}
\end{figure} 

With the formula for the gradient of $D_c{\tilde{f}}$, it is time for us to introduce the minimization algorithm. The basic idea is to find a discrete-time approximation of the continuous-time gradient flow
\begin{equation}
    \frac{d}{dt}\tilde{f_t}(i) = -\dfrac{\partial D_c}{\partial\tilde{f}(i)}(\tilde{f}_t)
\end{equation}
To achieve it, we utilize the exponential map to get a piecewise solution. We fix the time step as a small enough positive number $\tau$. For each positive integer $k$, we define
\begin{equation}
    \tilde{f}_{(k+1)\tau}(i) = \exp_{\tilde{f}_{k\tau}(i)}\bigg(-\tau \dfrac{\partial D_c}{\partial\tilde{f}(i)}(\tilde{f}_{k\tau})\bigg)
    \label{eq:exponential_map}
\end{equation}
For each interior vertex $i\in\tilde{V}$, the exponential map gives the position of the vertex $i$ at time $(k+1)\tau$. The update of the boundary vertices of $\tilde{V}$ requires more attention due to the constraint~\eqref{eq:hard_constraint}. For each vertex of $V$ that has more than $1$ representative in $\tilde{V}$, we need to choose only one representative $i\in\tilde{V}$ and get its updated position in time $(k+1)\tau$ using formula~\eqref{eq:exponential_map}. For any other representative $j$ such that $\pi_V(i)=\pi_V(j)$ and $j = \gamma(i)$ for some $\gamma\in \Gamma$, we update the position of $j$ using formula
\begin{equation}
    \tilde{f}_{(k+1)\tau}(j) = \gamma(\tilde{f}_{(k+1)\tau}(i)).
    \label{eq:boundary_update}
\end{equation}
In this way, we ensure that $\tilde{f}_{t}$ satisfies the hard constraint~\eqref{eq:hard_constraint} for any time $t>0$. 
Finally, we need to give some thresholds to determine when we stop the iterations. One solution is to stop the iteration when the energy no longer decreases, i.e. the iteration stops at the $n$-th step if
\begin{equation}
    \max_{i}d(\tilde{f}_{(n+1)\tau}(i), \tilde{f}_{n\tau}(i)) < \epsilon
\end{equation}
for some prescribed $\epsilon>0$, where $d$ is the geodesic distance in $(\D,g)$. We thus take $\tilde{f}_{n\tau}$ as the approximated hyperbolic harmonic map. Another solution is to stop the iteration when the gradient of $D_c$ on $\D^{|V|}$ is small enough, i.e., the iteration stops at the $n-$th step if
\begin{equation}
    \dfrac{1}{|V|}\sum_{i\in V} g \left( \dfrac{\partial D_c(\tilde{f})}{\partial\tilde{f}(i)},\dfrac{\partial D_c(\tilde{f})}{\partial\tilde{f}(i)} \right) < \epsilon
\end{equation}
for some prescribed $\epsilon>0$, where $g$ is the Riemannian metric on $\D$. The validity of this approach can be shown by a simple lemma
\begin{lemma}
    Let $M$ be a compact Riemannian manifold. Suppose $f$ is a smooth function on $M$ which has a unique critical point $x^*$. Then, for any open neighborhood $U$ of $x^*$, there exists an $\epsilon>0$ such that $||\nabla f(x)||\geq\epsilon$ for any $x\in M\backslash U$.
\end{lemma}
\begin{proof}
    Suppose that this statement is not correct. Then, we can take a sequence $(x_j)$ in $M\backslash U$ such that $||\nabla D_c(x_j)||$ tends to zero. Since $M\backslash U$ is compact, a subsequence of $(x_j)$ will converge to a critical point, leading to contradiction.
\end{proof}
To apply this lemma, we recall from the proof of \ref{thm:colin_de_verdiere} that $D_c$ is proper. So, we could restrict ourselves to a sub-level set of $D_c$ which is compact. Thus, if the gradient of $D_c$ at a map $f$ is small enough, we can safely tell that $f$ is close enough to the desired harmonic map.

\subsection{Convergence Rate}
It is always an interesting question to quest whether an iterative algorithm will converge and how fast it can be. To get an estimate for our algorithm, we need some result from Riemannian optimization. The following result from~\cite{boumal2023introduction} establishes an asymptotic convergence rate for constant step-size first order algorithms.
\begin{theorem}[Boumal \cite{boumal2023introduction}]
    Let $M$ be a Riemannian manifold with a retraction $R$. Let $f:M\to\R$ be a smooth function. Assume $x^*\in M$ satisfies
    \[\nabla f(x^*)=0\ \ \ \text{and}\ \ \ \operatorname{Hess}f(x^*)>0\]
    Let $0<\lambda_{\text{min}}\leq\lambda_{\text{max}}$ be the smallest and largest eigenvalues of $\operatorname{Hess}f(x^*)$, and let $\kappa=\frac{\lambda_{\text{max}}}{\lambda_{\text{min}}}$ denote the condition number of $\operatorname{Hess}f(x^*)$. Set $L>\frac{1}{2}\lambda_{\text{max}}$. Given $x_0\in M$, constant step-size Riemannian gradient descent iterates
    \[x_{k+1}=F(x_k),\ \ \ \text{with}\ \ \ F(x)=R_x\left(-\frac{1}{L}\nabla f(x)\right).\]
    These exists a neighborhood of $x^*$ such that, if the above sequence enters the neighborhood, then it stays in that neighborhood and it converges to $x^*$ at least linearly. If $L=\lambda_{\text{max}}$, the linear convergence factor is at most $1-1/\kappa$.
    \label{thm:numerical_convergence}
\end{theorem}
This theorem works for an arbitrary retraction on a Riemannian manifold, while the exponential map used by us is a special case of it. To see why this theorem applies to our algorithm, recall from the proof of~\ref{thm:colin_de_verdiere} that each geodesic realization $f:(V,E)\to S_2$ lifts to the universal cover and can be regarded as an element in the product space $\D^{|V|}$. In contrast, a geodesic realization can be recovered from an element in $\D^{|V|}$ through homotopy data. As long as the constraint~\eqref{eq:hard_constraint} is strictly satisfied, the initial map $\tilde{f}:(\tilde{V},\tilde{E})\to\D$ induces a well-defined geodesic realization $f:(V,E)\to\D/\Gamma$. With the homotopy class prescribed by the initialization, it suffices to implement the optimization algorithm simply in $\D^{|V|}$. The special update~\eqref{eq:boundary_update} for boundary vertices ensures that the map $\tilde{f}_t$ induces a well-defined geodesic realization $f_t:(V,E)\to\D/\Gamma$ for each iteration step. Since $\pi:\D\to\D/\Gamma$ is a covering map of Riemann surfaces, the formula for the gradient, exponential map, and Hessian can be safely used in the universal cover. Recall again from~\ref{thm:colin_de_verdiere} that the Hessian of $D_c$ at the discrete harmonic map is positive definite. As a result, the above theorem works well for our gradient descent algorithm, guaranteeing an at least linear convergence rate when the time step is small enough and the map in hand is close enough to the discrete harmonic map.

\section{Experiments}
\subsection{A genus 2 example}
We first demonstrate the numerical results of the discrete harmonic map from a genus-2 triangulated surface $S_1$ to the hyperbolic surface $S_2$ constructed as in Fig.~\ref{fig:hyperbolic_octagon}. The initial closed triangulated surface is shown in Fig.~\ref{fig:g2_original}. As illustrated above, the first step of our algorithm is to compute a flat hyperbolic metric $l$ on $S_1$. We then cut $S_1$ along a homology basis to get a simply-connected triangulated surface $\tilde{S}_1$ and embed $\tilde{S}_1$ in $\H$ with respect to $l$, as shown in Fig.~\ref{fig:g2_embedding}. Fig.~\ref{fig:g2_initial} shows the Euclidean harmonic map from $\tilde{S}_1$ to the octagon, which is obtained by solving a Dirichlet problem. Initialized by this map, we run our gradient descent algorithm with a fixed step-size. Fig.~\ref{fig:g2_final} shows the output harmonic map $\tilde{f}:(\tilde{V},\tilde{E})\to\D$ with~\eqref{eq:hard_constraint} satisfied.

\begin{figure*}[t!]
\centering
\begin{subfigure}[t]{0.47\textwidth}
\centering
\includegraphics[width=.9\linewidth]{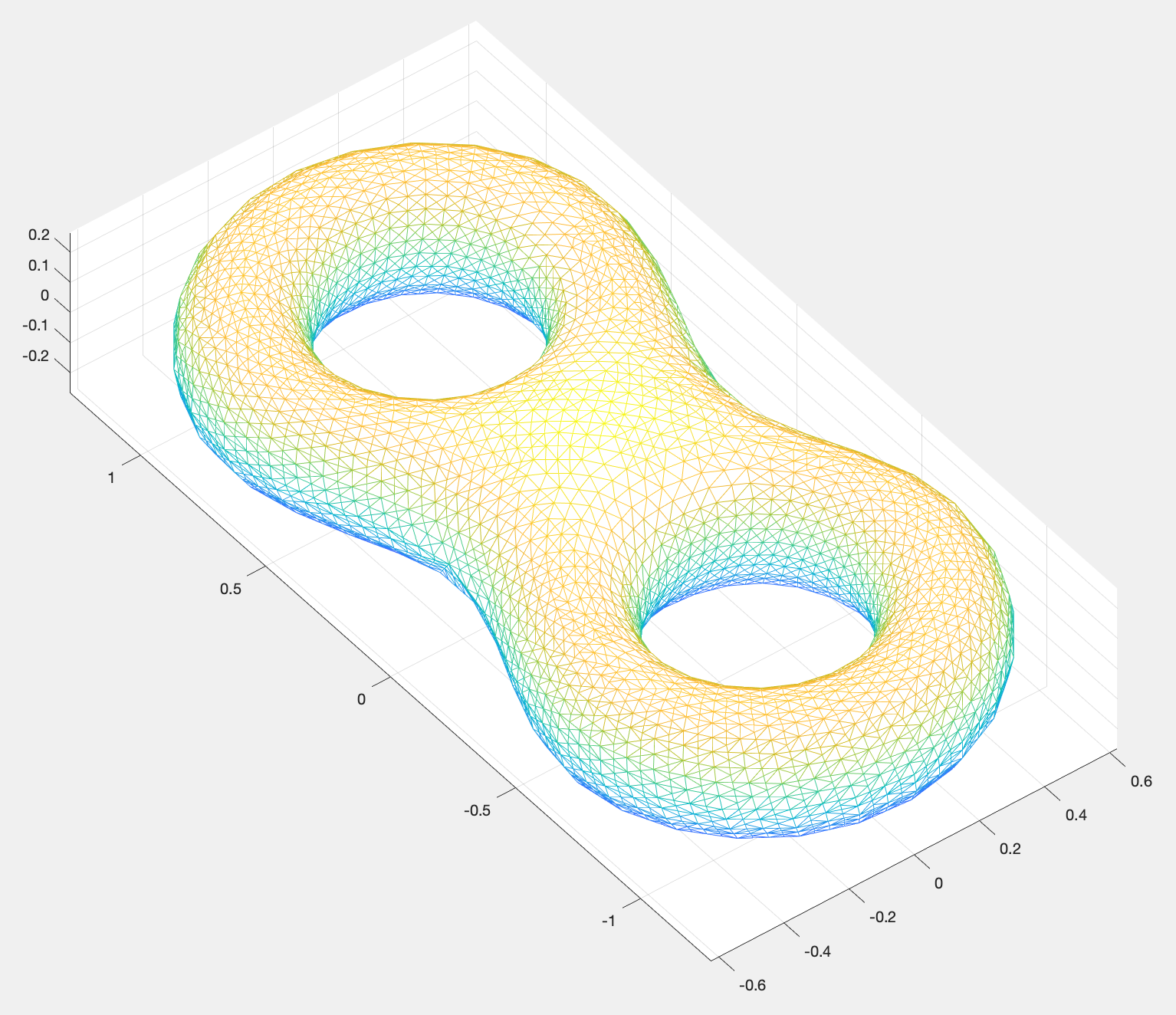}
\caption{The original genus-2 surface}
\label{fig:g2_original}
\end{subfigure}%
\begin{subfigure}[t]{0.5\textwidth}
\centering
\includegraphics[width=.9\linewidth]{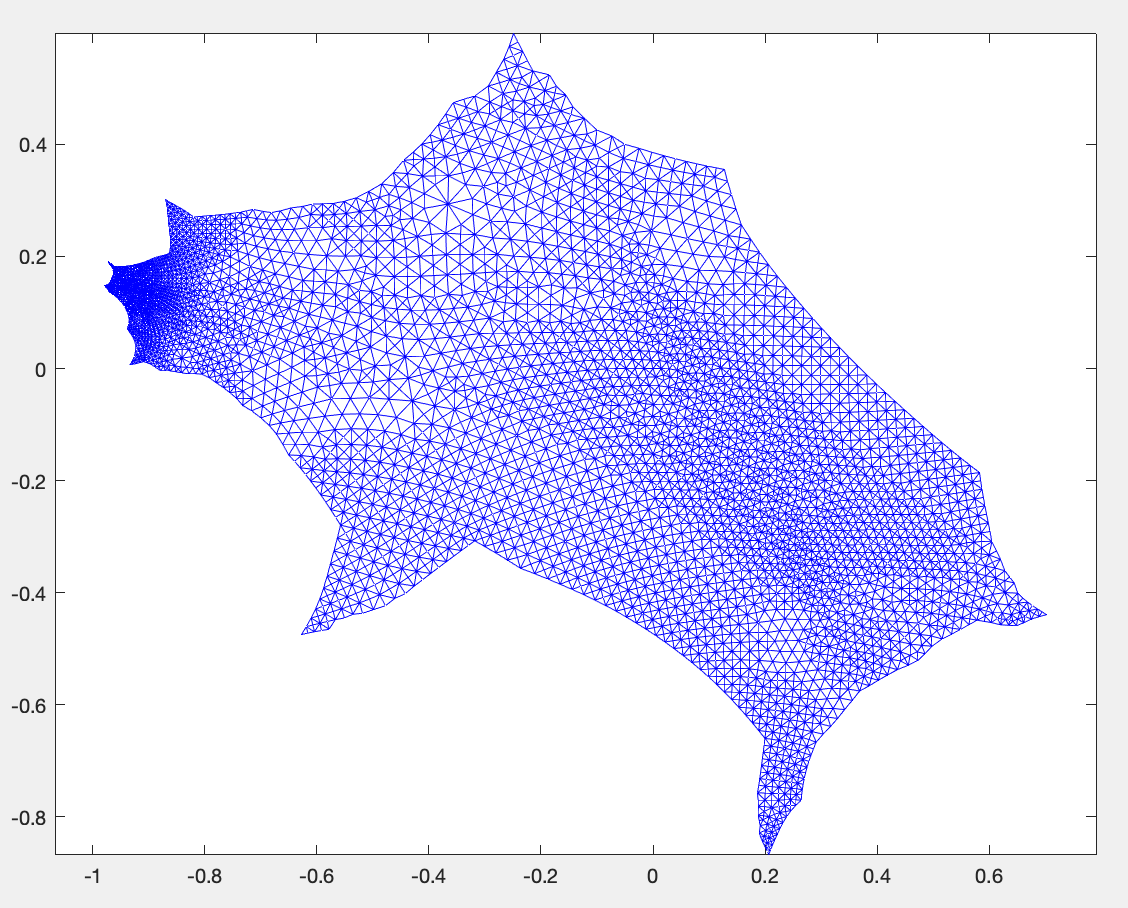}
\caption{An embedding given by the discrete Ricci flow}
\label{fig:g2_embedding}
\end{subfigure}
\begin{subfigure}[t]{.5\textwidth}
\centering
\includegraphics[width=.9\linewidth]{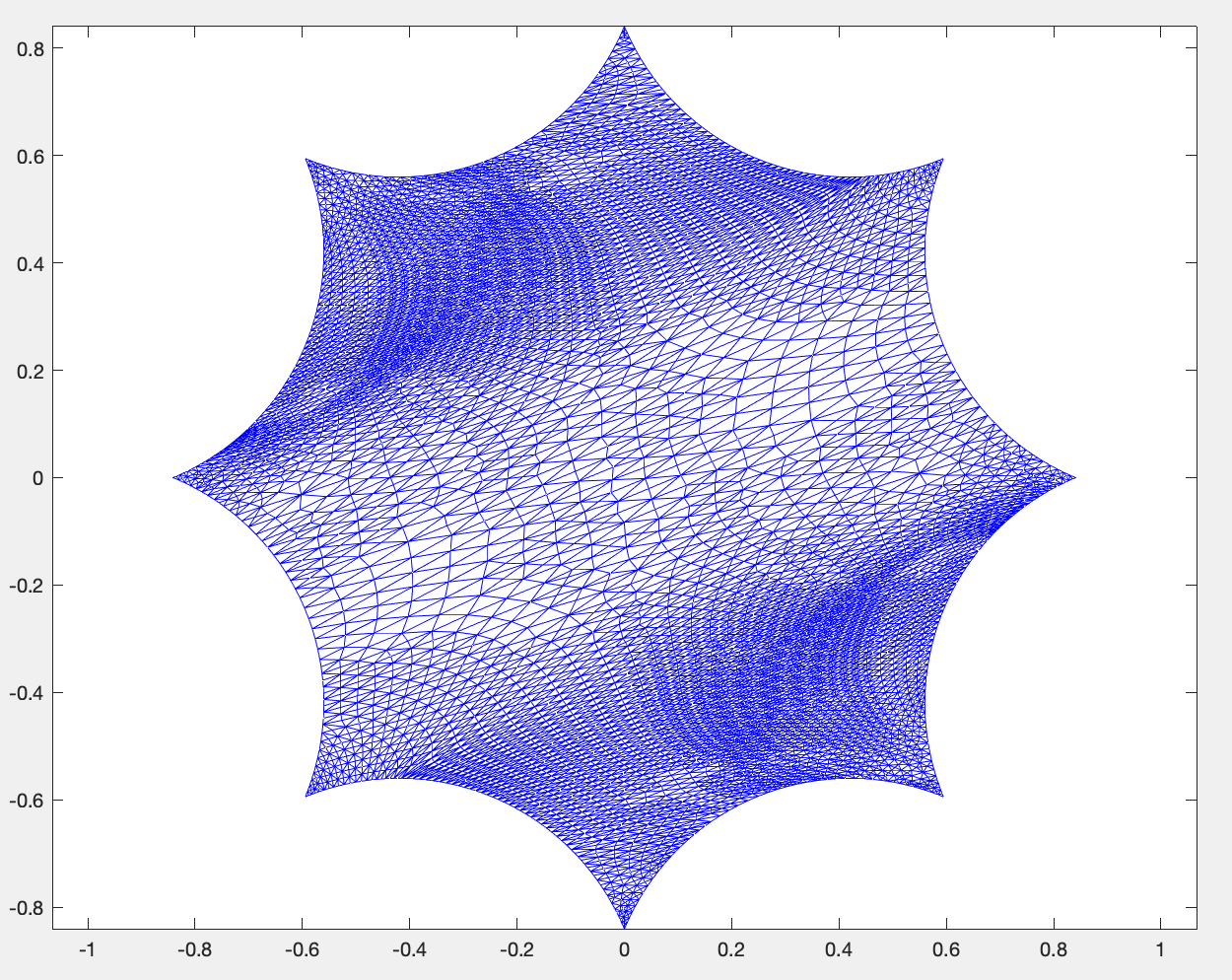}
\caption{The initial map}
\label{fig:g2_initial}
\end{subfigure}%
\begin{subfigure}[t]{.5\textwidth}
\centering
\includegraphics[width=.9\linewidth]{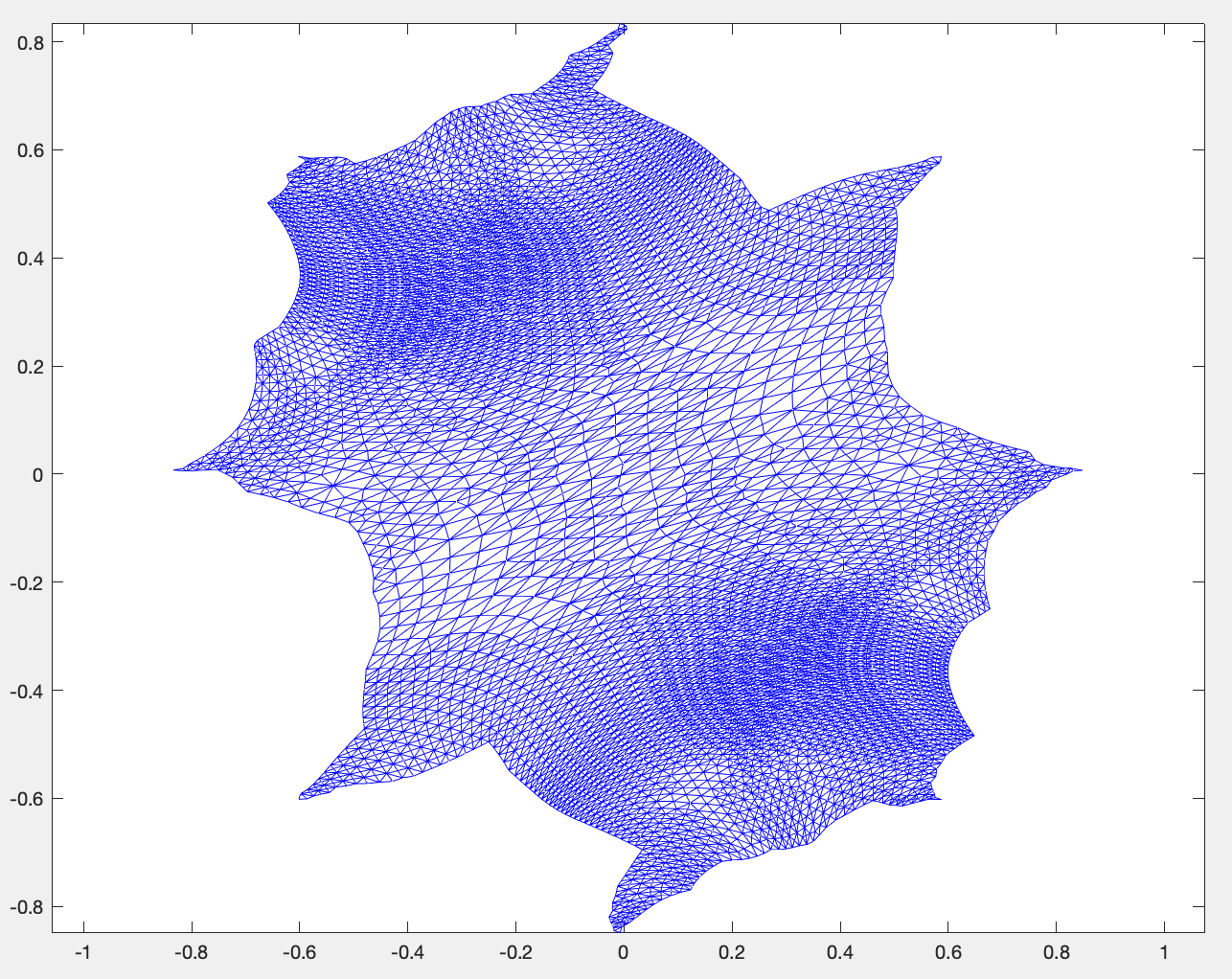}
\caption{The output harmonic map}
\label{fig:g2_final}
\end{subfigure}%
\caption{A genus-2 example}
\label{fig:g2_example_1}
\end{figure*}

Fig.~\ref{fig:g2_tessellation} shows the tessellation of it under the action of several side-pairings, which proves that the output map strictly satisfies the constraint~\eqref{eq:hard_constraint}. Fig.~\ref{fig:g2_energy} and Fig.~\ref{fig:g2_gradient} show the descent of the Dirichlet energy and the norm of the gradient as time goes on, which is discretized at the product of the step size and number of iterations. The last figure shows the convergence rate of the descent, whose value at iteration $k$ is defined as
\begin{equation}
    \big(D_c(\tilde{f}_{k\tau})-D_c(\tilde{f})\big)/\big(D_c(\tilde{f}_{(k-1)\tau})-D_c(\tilde{f})\big)
\end{equation}
where $\tau$ is the step size and $\tilde{f}$ is the final map. As can be seen, the convergence rate rapidly decays from $1$ as the iteration approaches the end. This numerical experiment shows that we could expect an at least linear asymptotic convergence rate for some appropriate step size, which is theoretically stated in \ref{thm:numerical_convergence}.

\begin{figure}
\begin{subfigure}[t]{.51\textwidth}
\centering
\includegraphics[width=.9\linewidth]{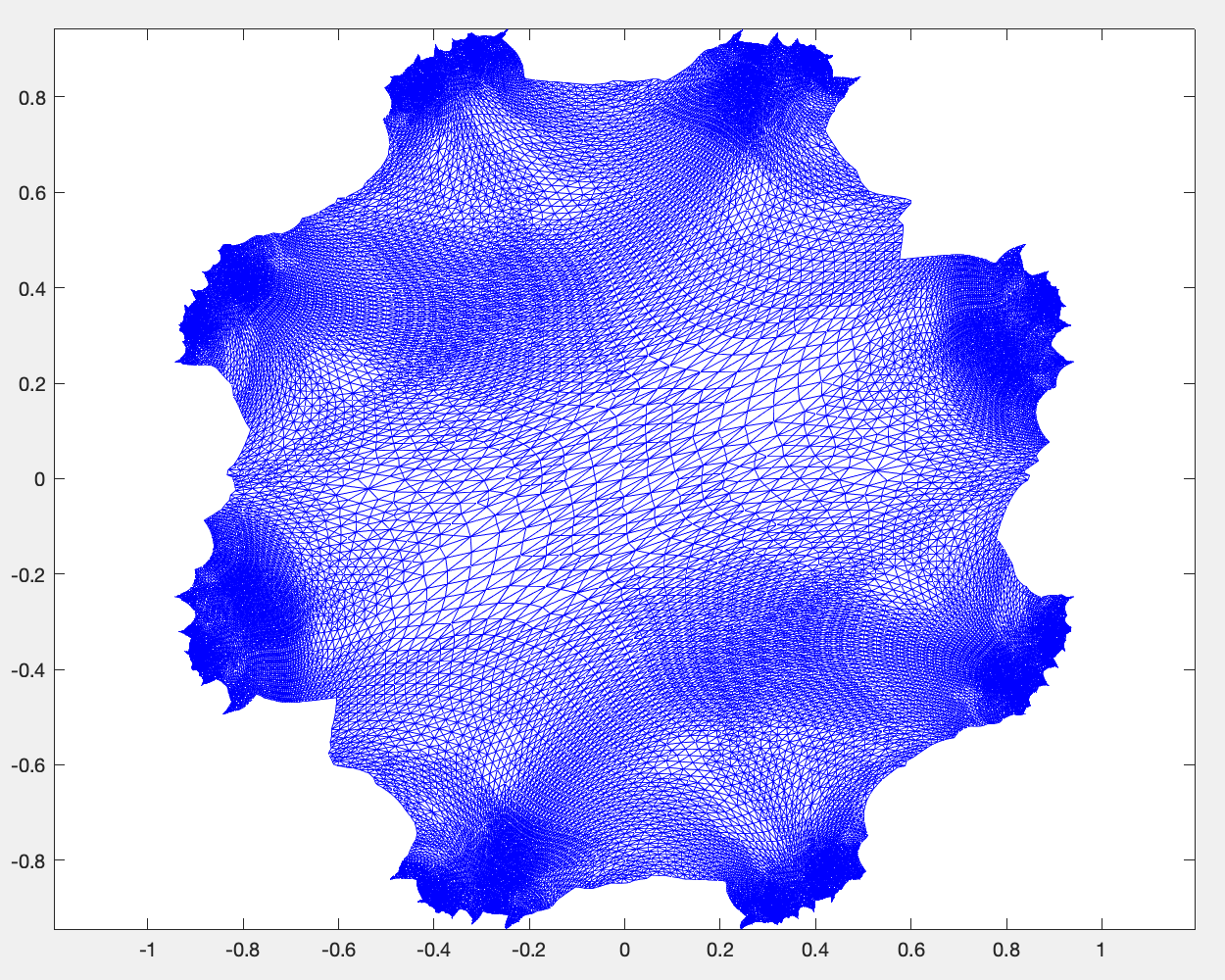}
\caption{Tessellation given by output map}
\label{fig:g2_tessellation}
\end{subfigure}%
\begin{subfigure}[t]{.49\textwidth}
\centering
\includegraphics[width=.9\linewidth]{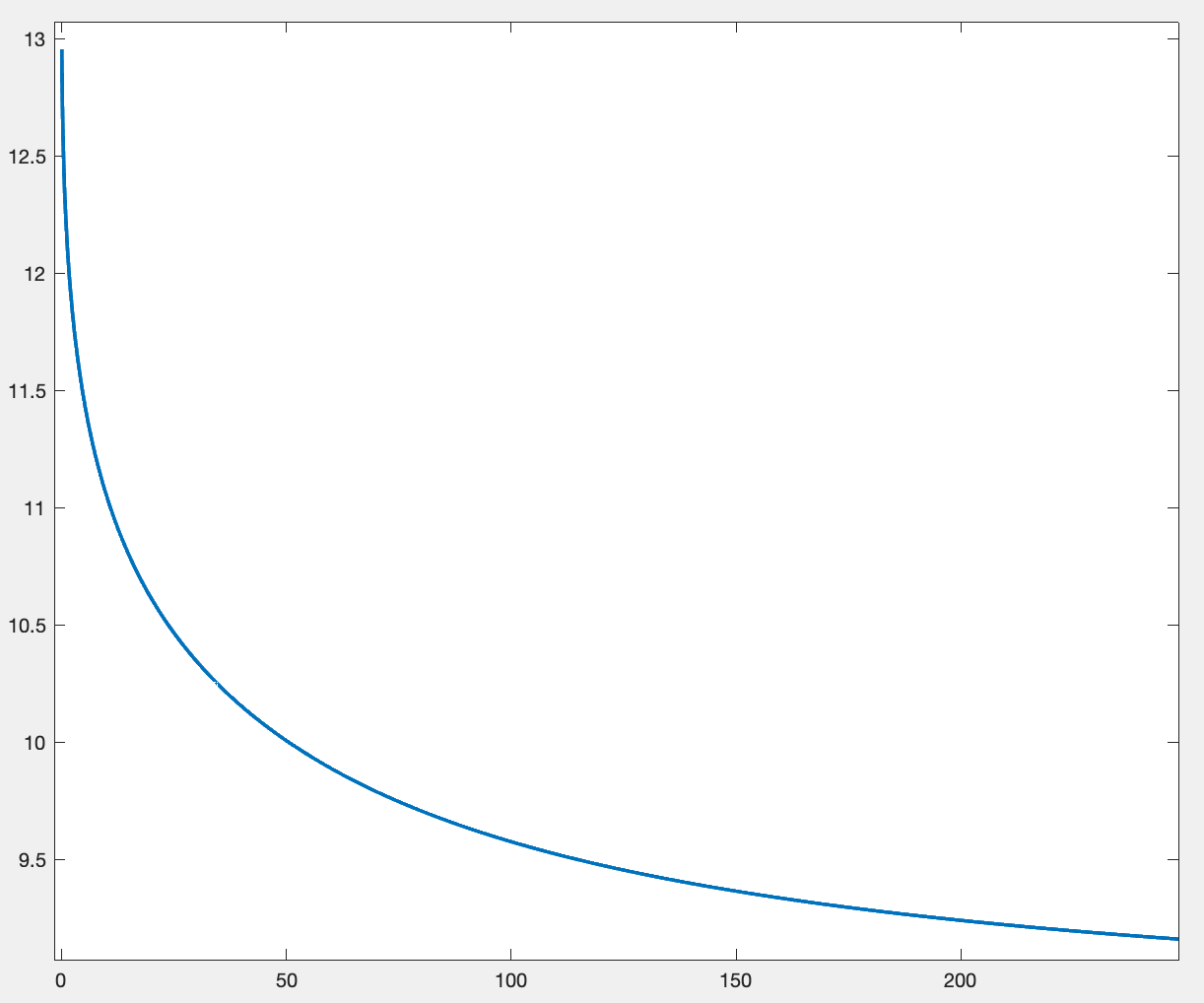}
\caption{The descent of energy}
\label{fig:g2_energy}
\end{subfigure}%
\\
\begin{subfigure}[t]{.5\textwidth}
\centering
\includegraphics[width=.9\linewidth]{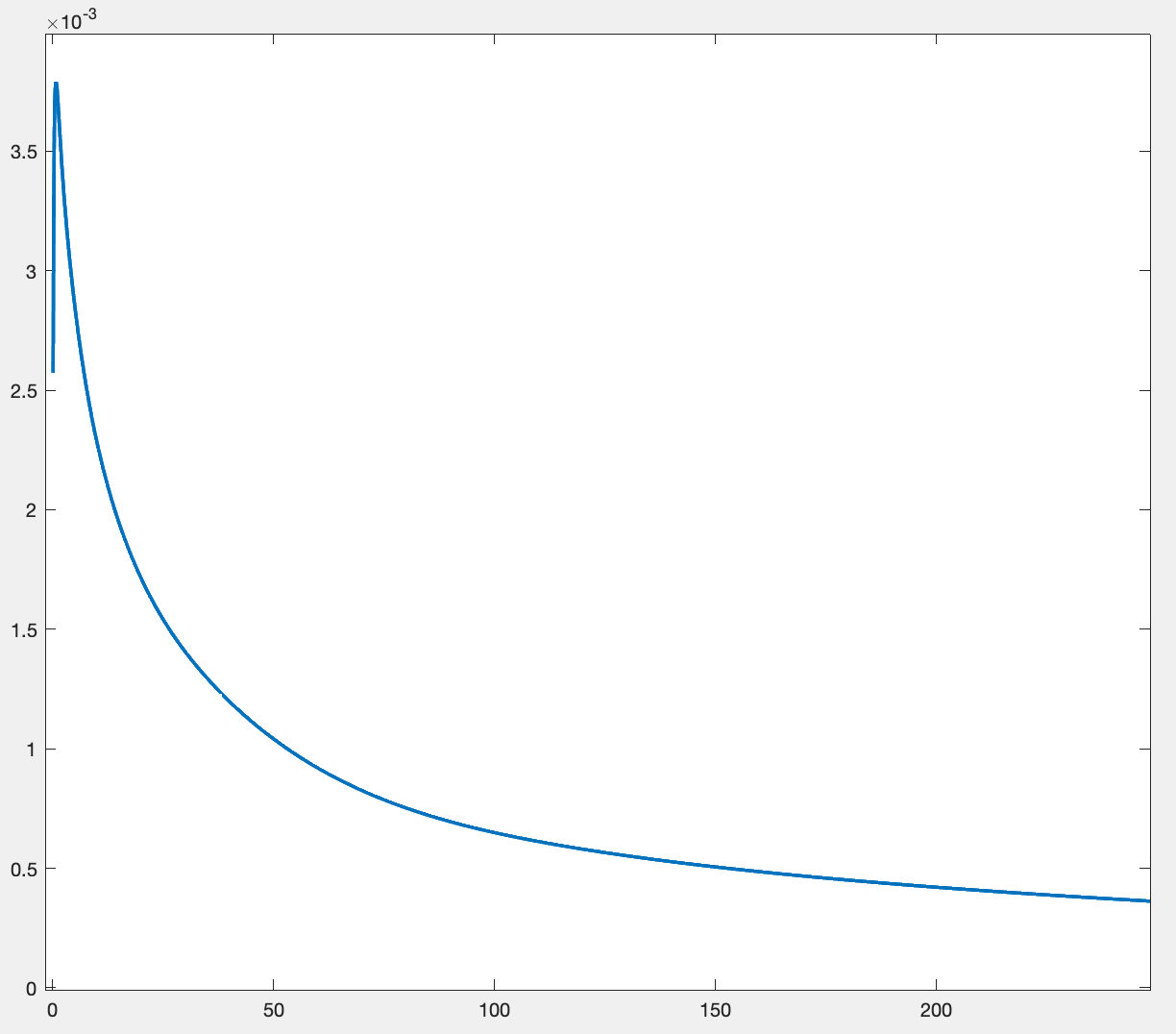}
\caption{The descent of gradient}
\label{fig:g2_gradient}
\end{subfigure}%
\begin{subfigure}[t]{.5\textwidth}
\centering
\includegraphics[width=.9\linewidth]{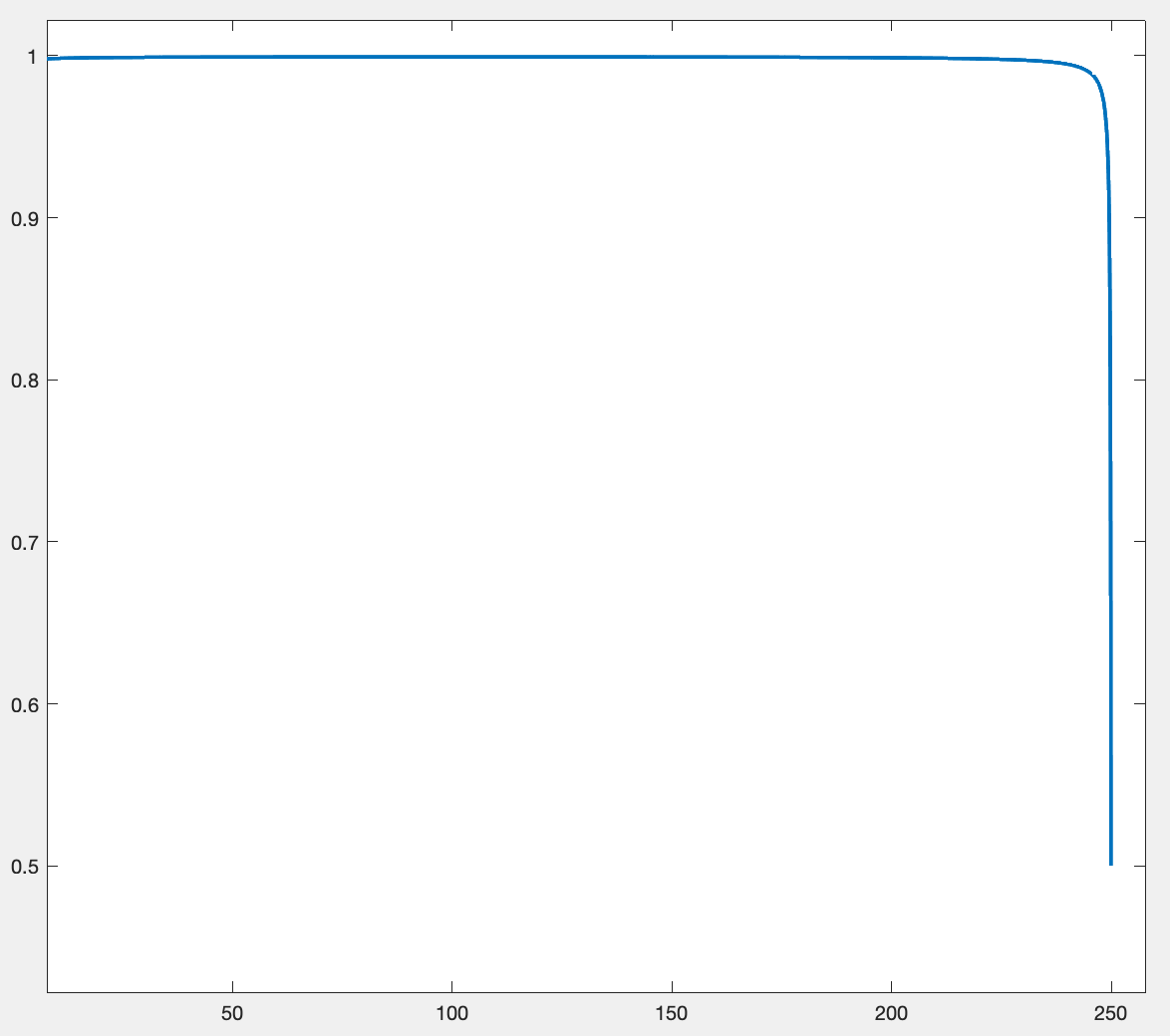}
\caption{Convergence rate}
\label{fig:g2_convergence_rate}
\end{subfigure}%
\caption{A genus-2 example}
\label{fig:g2_example_2}
\end{figure}

\subsection{A genus-3 example}
Next, we demonstrate the numerical results of the discrete harmonic map from a genus-3 surface to a hyperbolic surface whose fundamental domain is the standard 12-gon. Again, Fig.~\ref{fig:g3_original} shows the original surface. Fig.~\ref{fig:g3_polygon} shows the canonical polygon given by the uniformization metric. Fig.~\ref{fig:g3_initial_local} shows a local region of the initial Euclidean harmonic map~\ref{fig:g3_initial}, whereas Fig.~\ref{fig:g3_final_local} shows the same region of the output harmonic map~\ref{fig:g3_final}. As we can see, the initial Euclidean harmonic map is not bijective but the output hyperbolic map is. With a fixed step size that is small enough, Fig.~\ref{fig:g3_example_3} shows the descent of the energy and the norm of the gradient, and the convergence rate. Similar to the genus-2 case, the convergence is good when we take this step size. For comparison, Fig.~\ref{fig:g3_example_4} shows the descent of the energy and the norm of the gradient with respect to a step size that is $20$ times the previous one. As we can see, the energy oscillates strongly over the iterations, while the norm of the gradient oscillates a little. However, the two quantities still decrease as time passes. 

\begin{figure}
\begin{subfigure}[t]{.49\textwidth}
\centering
\includegraphics[width=.9\linewidth]{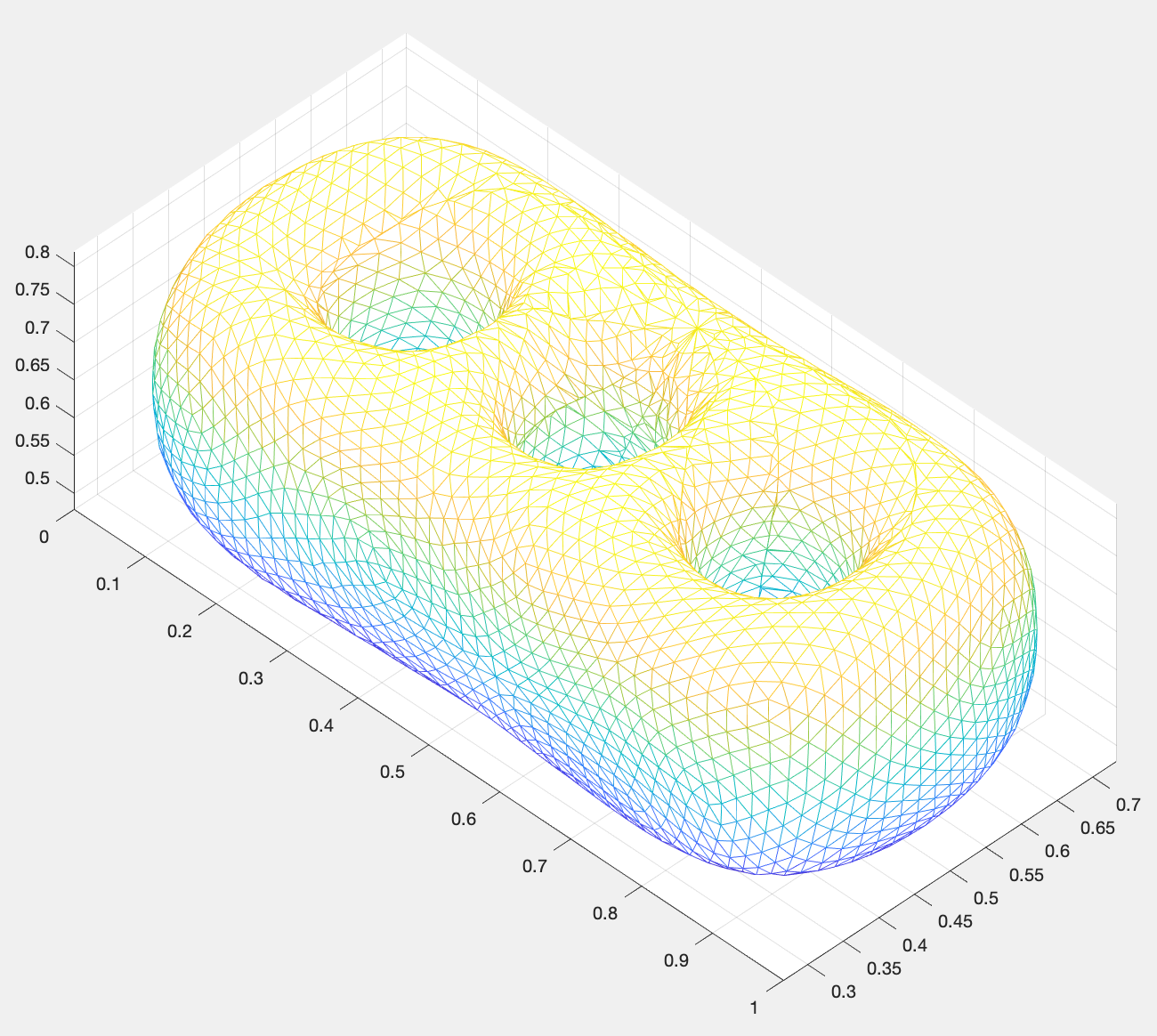}
\caption{The original genus-3 surface}
\label{fig:g3_original}
\end{subfigure}%
\begin{subfigure}[t]{.51\textwidth}
\centering
\includegraphics[width=.9\linewidth]{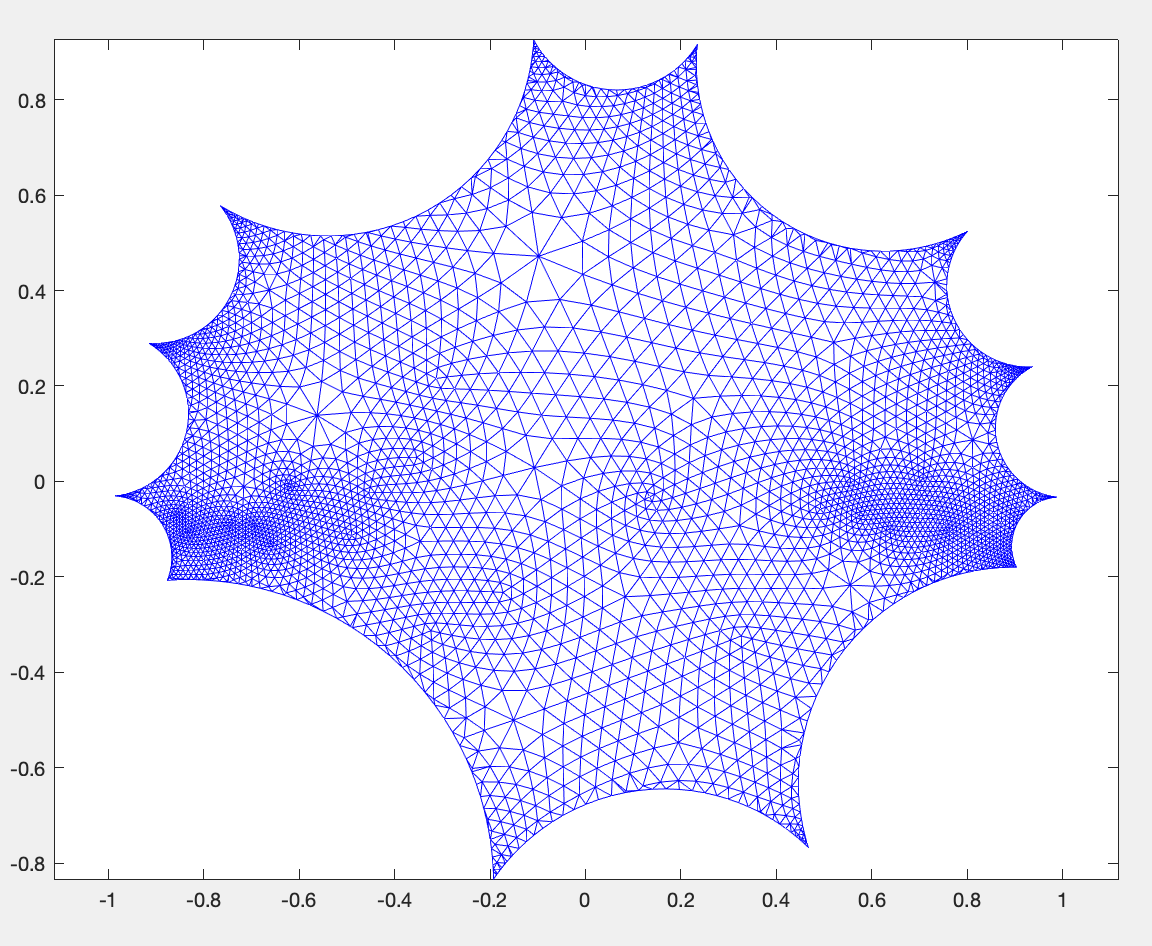}
\caption{The canonical polygon}
\label{fig:g3_polygon}
\end{subfigure}%
\caption{A genus-3 example}
\label{fig:g3_example_1}
\end{figure}

\begin{figure}
\begin{subfigure}[t]{.5\textwidth}
\centering
\includegraphics[width=.9\linewidth]{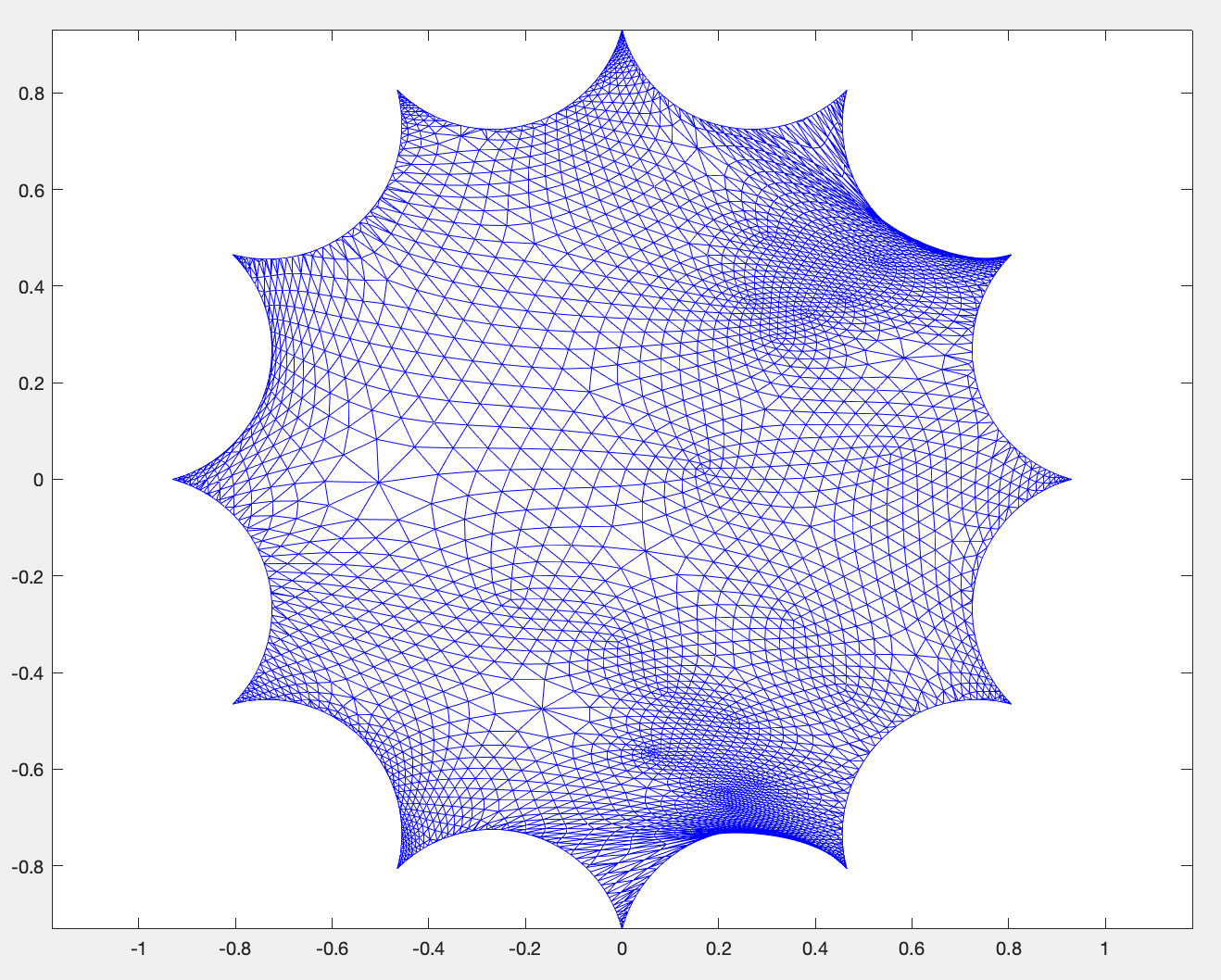}
\caption{The initial map}
\label{fig:g3_initial}
\end{subfigure}%
\begin{subfigure}[t]{.5\textwidth}
\centering
\includegraphics[width=.9\linewidth]{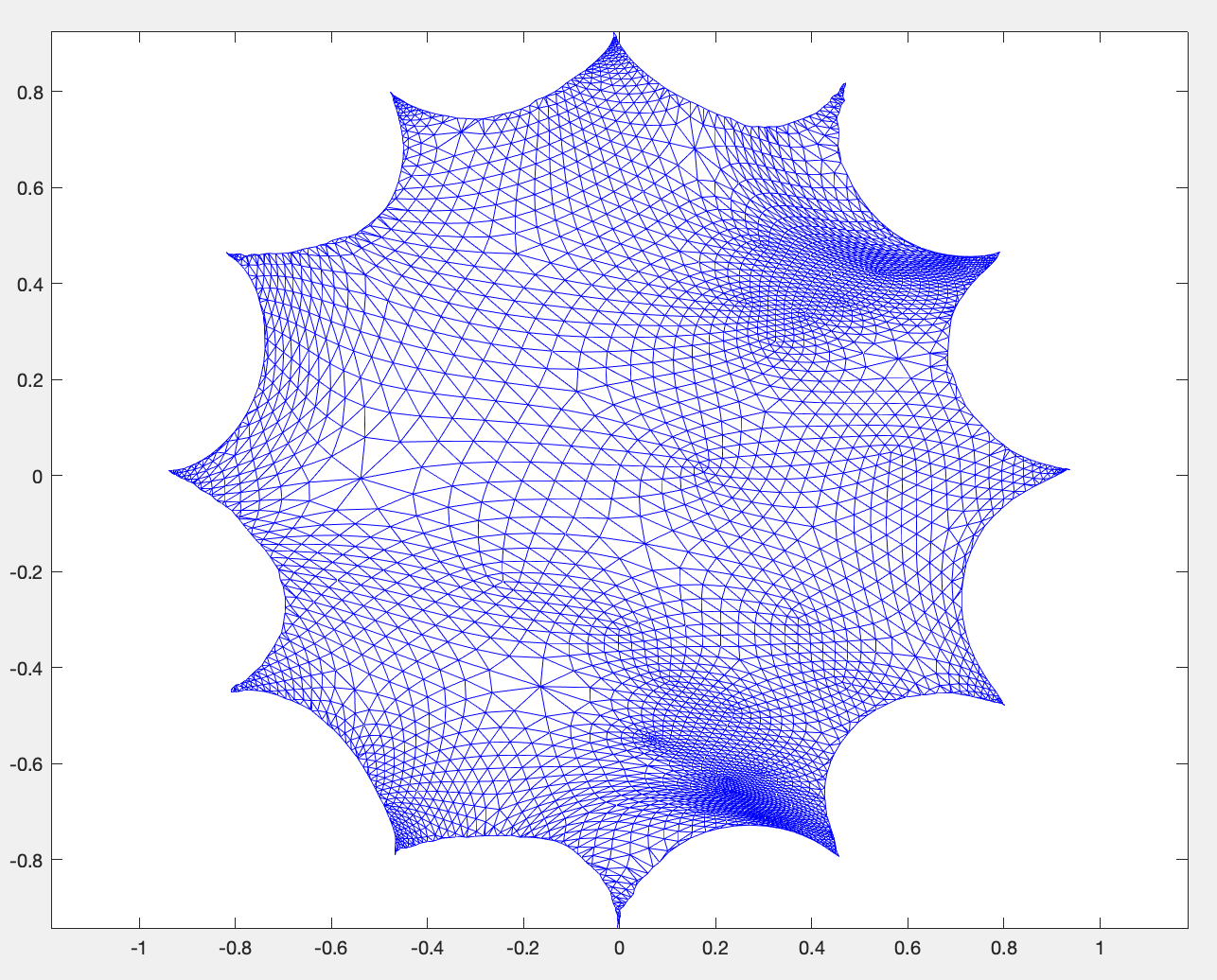}
\caption{The output harmonic map}
\label{fig:g3_final}
\end{subfigure}%
\\
\begin{subfigure}[t]{.5\textwidth}
\centering
\includegraphics[width=.8\linewidth]{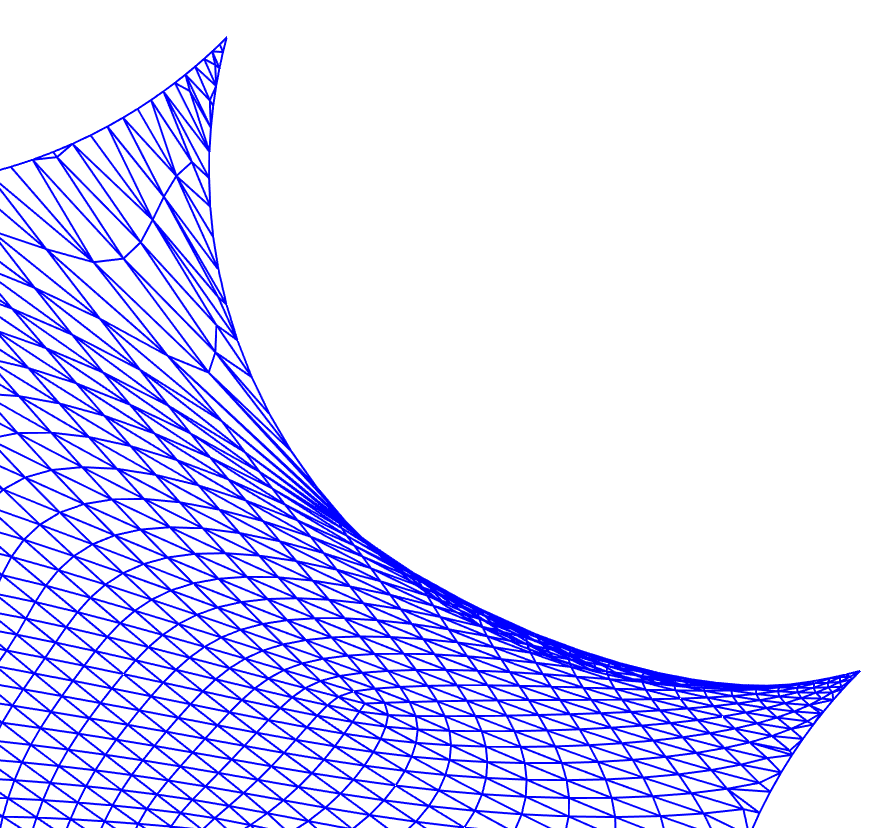}
\caption{A local region of initial map}
\label{fig:g3_initial_local}
\end{subfigure}%
\begin{subfigure}[t]{.5\textwidth}
\centering
\includegraphics[width=.8\linewidth]{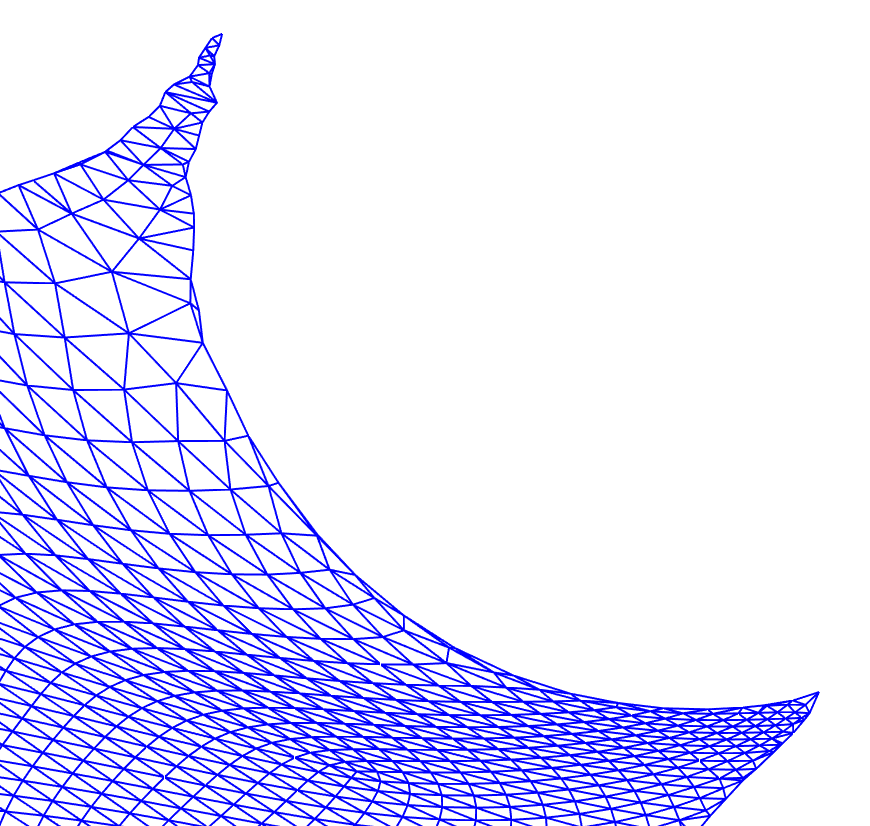}
\caption{A local region of output map}
\label{fig:g3_final_local}
\end{subfigure}%
\caption{A genus-3 example}
\label{fig:g3_example_2}
\end{figure}

\begin{figure}
\begin{subfigure}[t]{.51\textwidth}
\centering
\includegraphics[width=.9\linewidth]{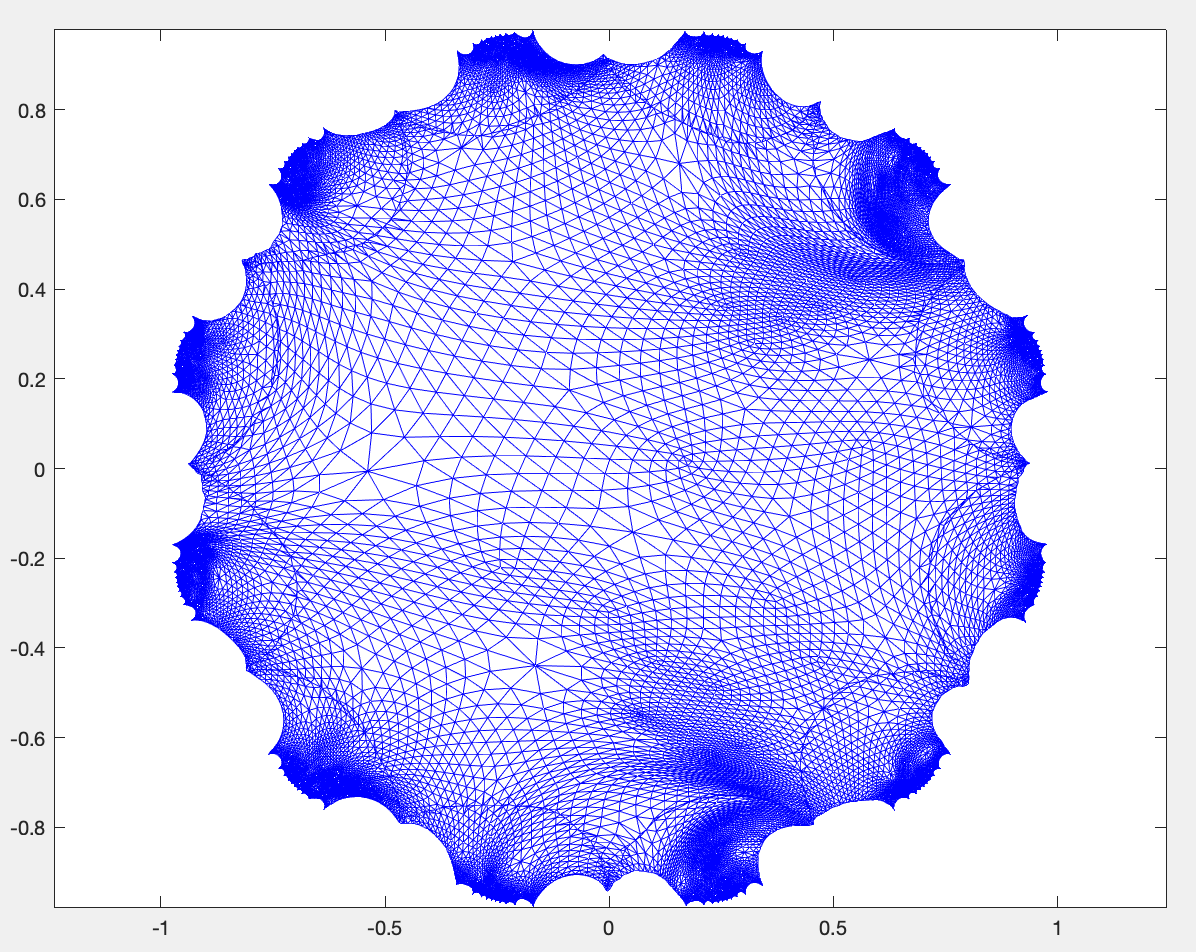}
\caption{Tessellation given by output map}
\label{fig:g3_tessellation}
\end{subfigure}%
\begin{subfigure}[t]{.49\textwidth}
\centering
\includegraphics[width=.9\linewidth]{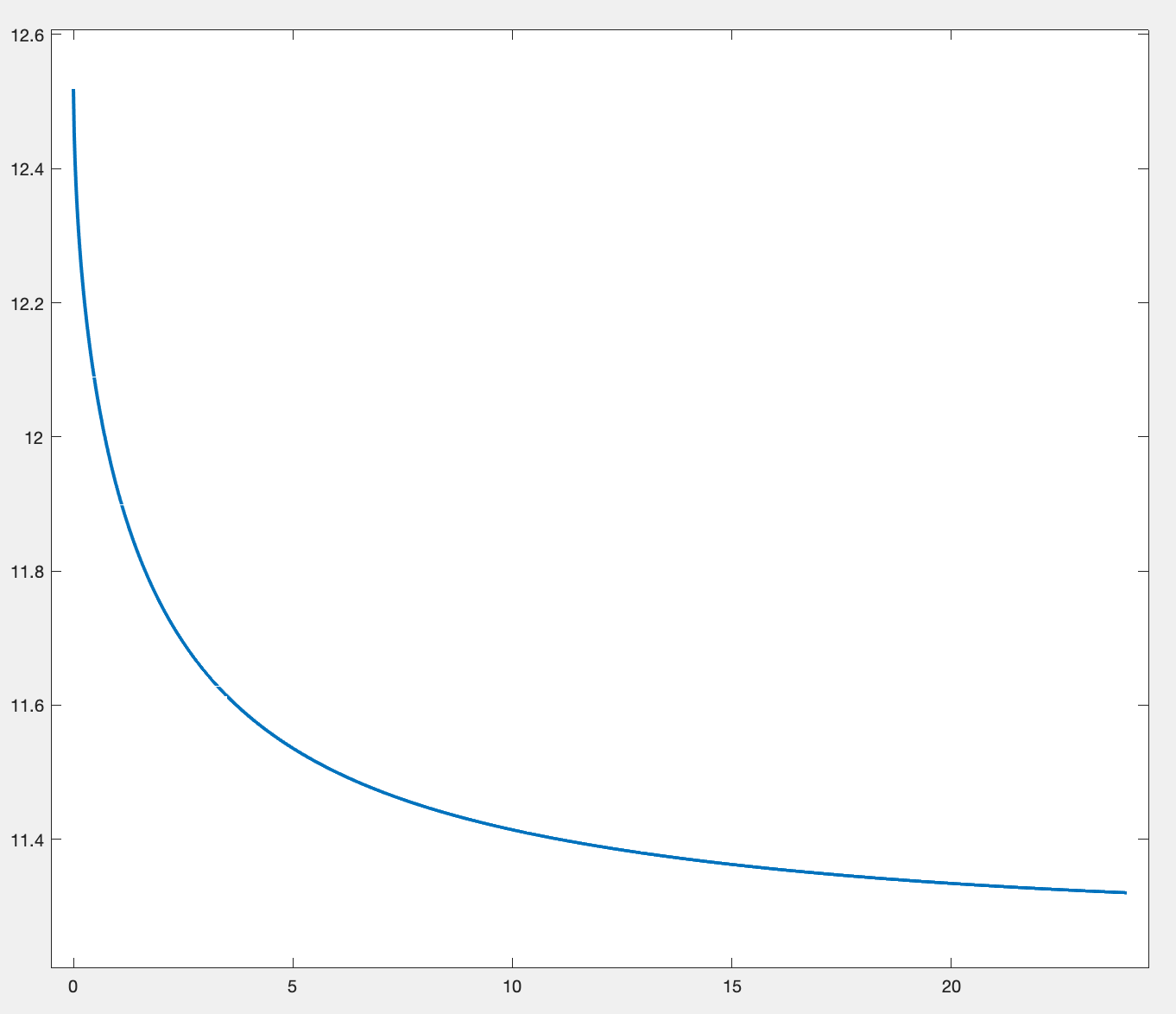}
\caption{The descent of energy}
\label{fig:g3_energy}
\end{subfigure}%
\\
\begin{subfigure}[t]{.5\textwidth}
\centering
\includegraphics[width=.9\linewidth]{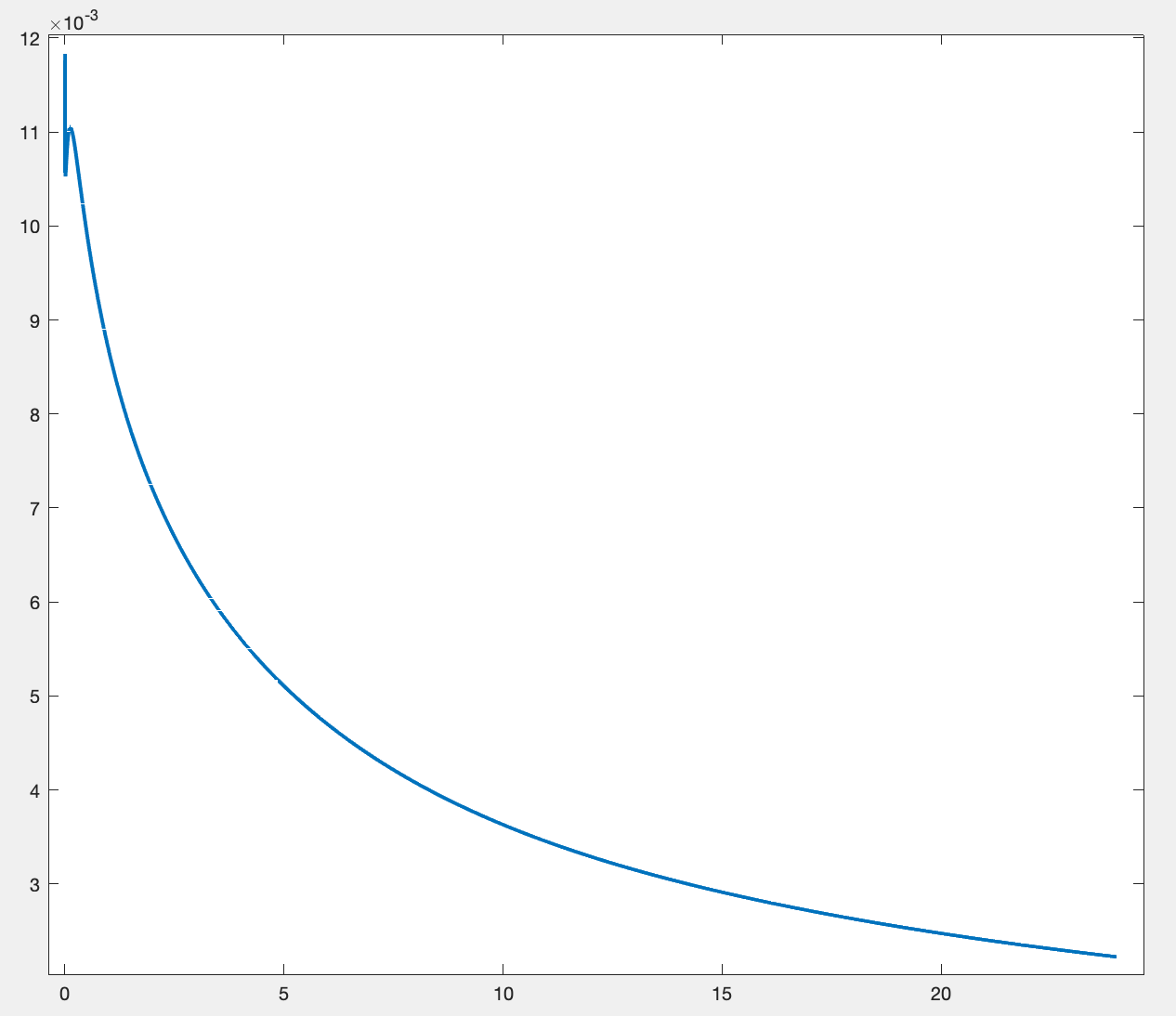}
\caption{The descent of gradient}
\label{fig:g3_gradient}
\end{subfigure}%
\begin{subfigure}[t]{.5\textwidth}
\centering
\includegraphics[width=.9\linewidth]{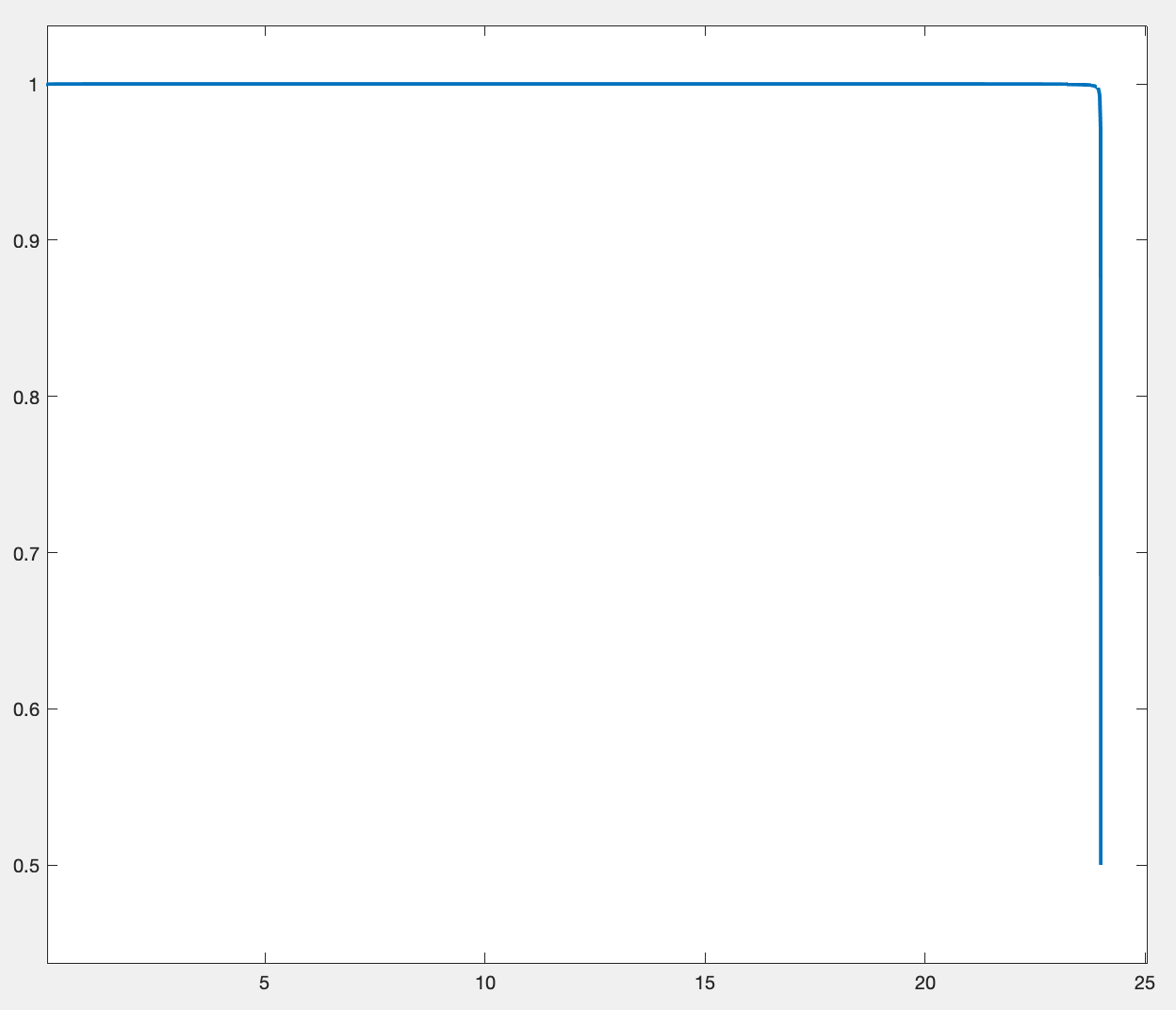}
\caption{Convergence rate}
\label{fig:g3_convergence_rate}
\end{subfigure}%
\caption{A genus-3 example}
\label{fig:g3_example_3}
\end{figure}

\begin{figure}
\begin{subfigure}[t]{.5\textwidth}
\centering
\includegraphics[width=.9\linewidth]{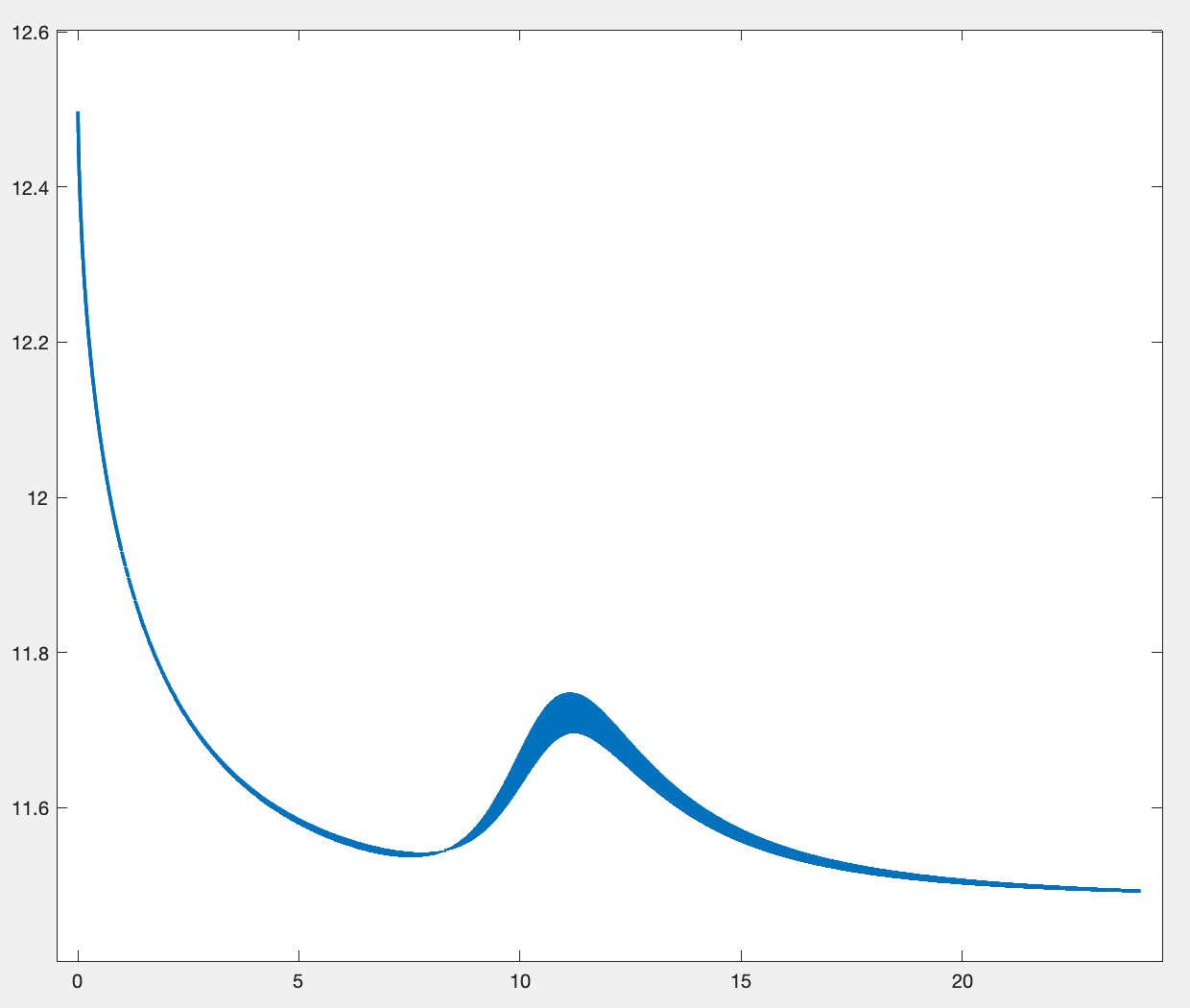}
\caption{Descent of energy for a larger step}
\label{fig:g3_energy_dt_larger}
\end{subfigure}%
\begin{subfigure}[t]{.5\textwidth}
\centering
\includegraphics[width=.9\linewidth]{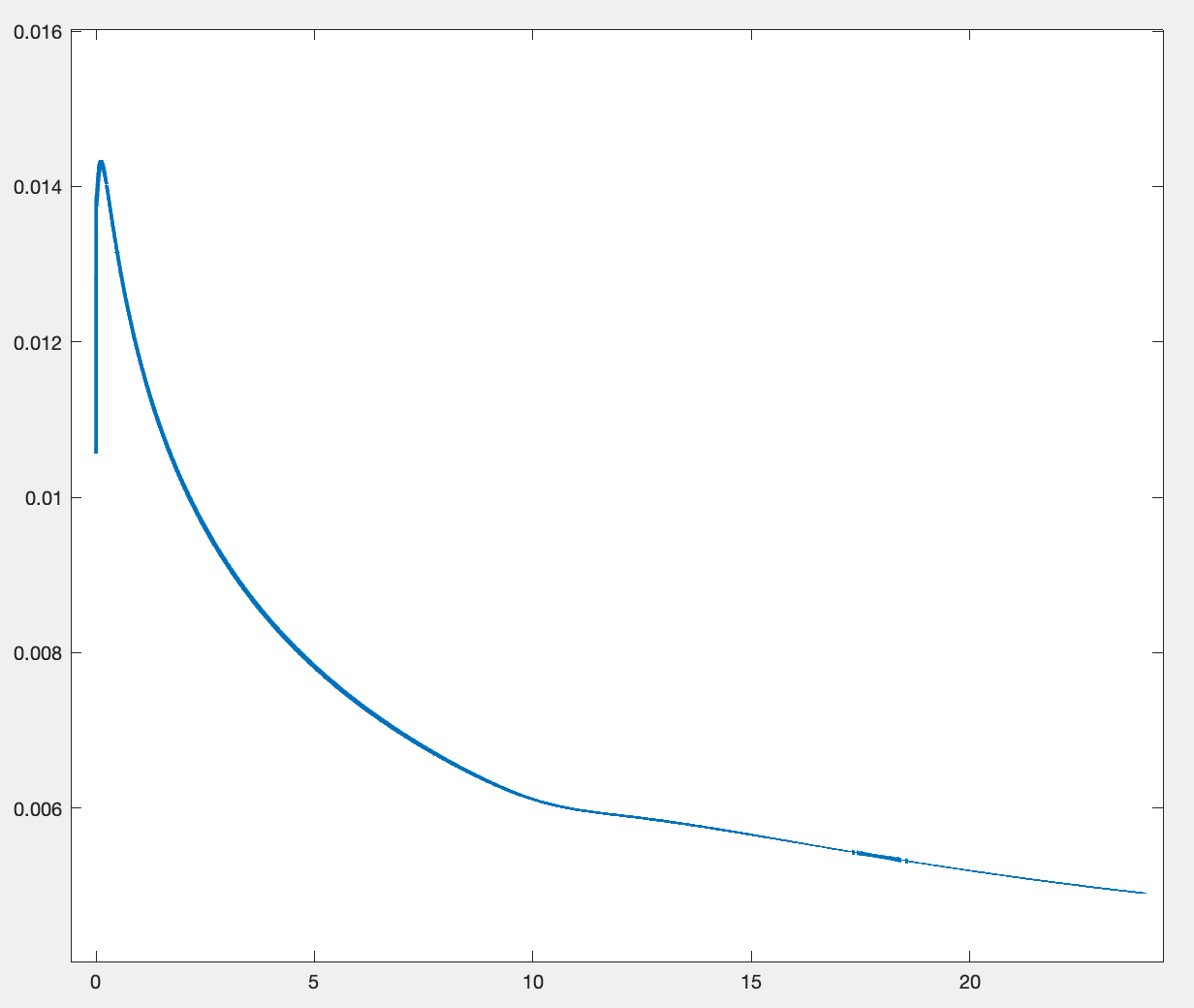}
\caption{Descent of gradient for a larger step}
\label{fig:g3_gradient_dt_larger}
\end{subfigure}%
\caption{A genus-3 example}
\label{fig:g3_example_4}
\end{figure}

\section{An Application to Remeshing}
The discrete harmonic map provides an effective way to map a high-genus surface onto a parameter domain globally. Hence, to process the original surface, we could choose to work on the parameter domain first, and then retrieve the desired effects on the surface by the inverse mapping. An important mesh-processing task is remeshing, which is crucial in scientific computing and computer graphics. In the Euclidean case, we usually generate a mesh structure in the parameter domain and map it back to the surface by the interpolated inverse mapping. The problem becomes more intricate for high-genus surfaces. The reason is that to compute the harmonic map, the original surface need to be cut to be simply-connected, and the final map is represented as the developing map from the simply-connected one to a fundamental domain of the target surface in $\D$.  

Let us denote the original triangulated surface $S$ by $(V,E,F)$ and the simply-connected one $\tilde{S}$ by $(\tilde{V},\tilde{E},F)$. Also, we suppose that the target surface is $\D/\Gamma$ for some Fuchsian group $\Gamma$ and one of its fundamental domain is a hyperbolic polygon $P$. A new mesh structure shall be generated on the polygon $P$ such that the boundary condition~\eqref{eq:hard_constraint} must be strictly met. With a harmonic map $\tilde{f}:\tilde{S}\to\D$ on hand, we first notice that $\tilde{f}(\tilde{S})$ does not necessarily cover $P$. But since $P$ is compact, we could find a discrete subset $\{\gamma_1,\cdots,\gamma_k\}\subset \Gamma$ such that
\begin{equation}
    P\subset\bigcup_{i=1}^{k}\gamma_i(\tilde{f}(\tilde{S}))
\end{equation}
Recall the fact that for any $\gamma\in \Gamma$, we have $\pi_V(i)=\pi_V(\gamma(i))$ for each vertex $i\in\tilde{V}$, where $\pi_V$ is the canonical projection from $\tilde{V}$ to $V$. Thus, the map $\pi_V$ can be properly defined and interpolated on $\bigcup_{i=1}^{k}\gamma_i(\tilde{f}(\tilde{S}))$, which gives us a way to map the mesh on $P$ back to the original surface. By gluing the boundary edges of this mesh, we finally get a new mesh structure on the original surface. This remeshing method is global and does not rely on partitioning the original mesh into several simply-connected pieces. But near the base point of the homology basis, the geometry may be complicated. Thus, some local remeshing near the base point may be taken as the last step to secure a good enough mesh structure.

\begin{figure}
\begin{subfigure}[t]{.4\textwidth}
\centering
\includegraphics[width=.8\linewidth]{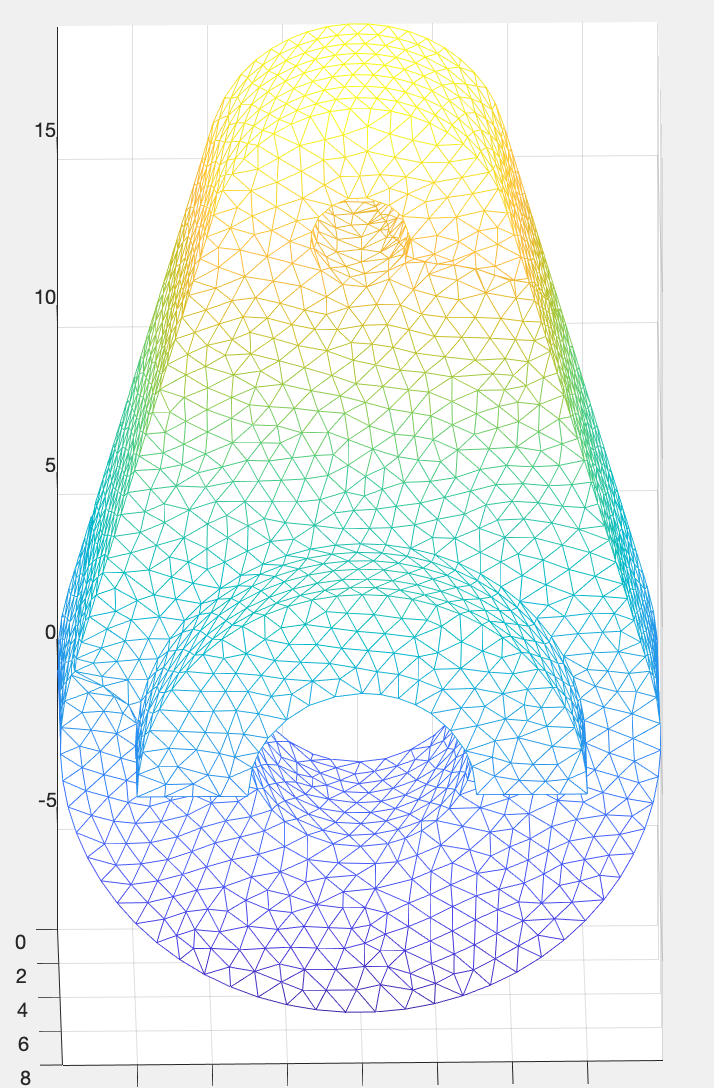}
\caption{A given genus-2 surface}
\label{fig:remesh_original}
\end{subfigure}%
\begin{subfigure}[t]{.6\textwidth}
\centering
\includegraphics[width=.9\linewidth]{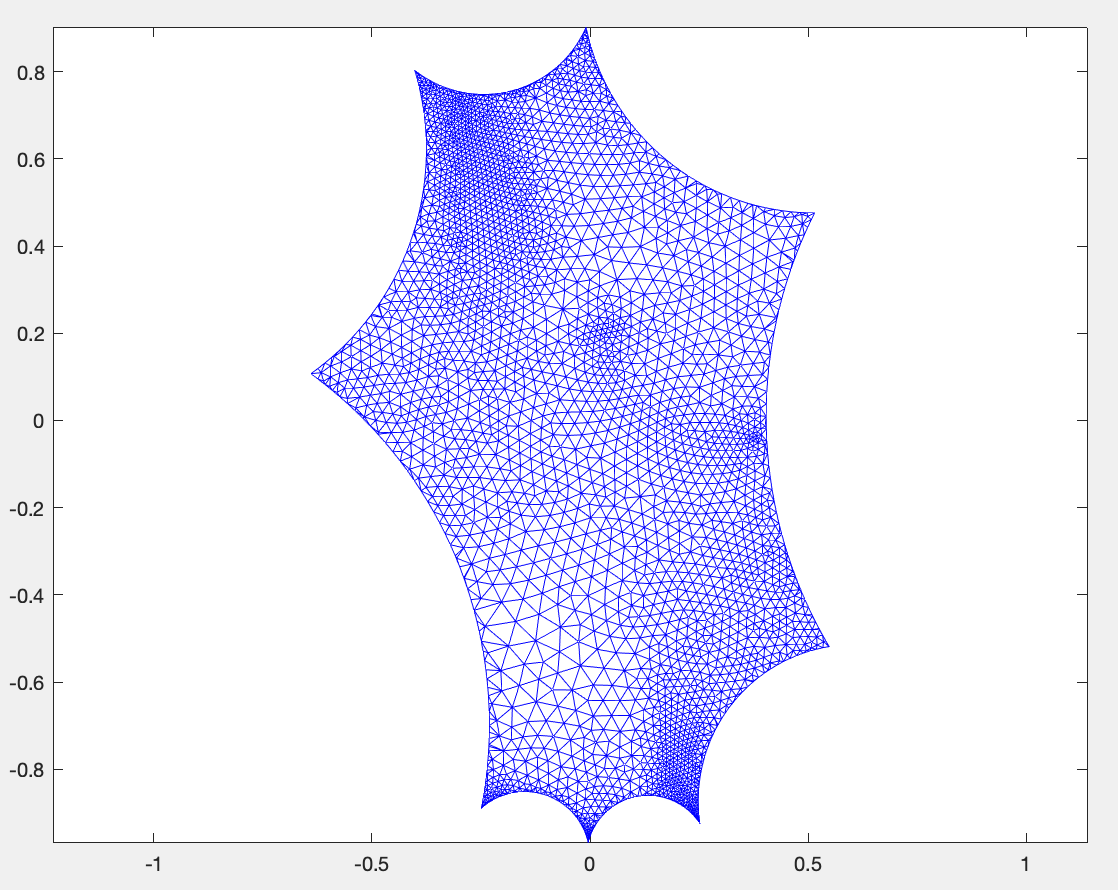}
\caption{The canonical polygon}
\label{fig:remesh_polygon}
\end{subfigure}%
\caption{The given surface to be remeshed}
\label{fig:remesh_1}
\end{figure}

To choose the target surface, one approach is to directly construct a paired polygon such as the octagon shown in Fig~\ref{fig:hyperbolic_octagon}. Another approach is to choose another triangulated surface of the same topology, compute a uniformization metric on it using the discrete Ricci flow, and then compute the mapping from the triangulated surface to the canonical polygon in $\D$ with respect to this flat metric. Here, we demonstrate the numerical result of both cases. Fig.~\ref{fig:remesh_1} shows a triangulated surface that has been cut along a homology basis. We find this surface in the online database Thingi10K~\cite{zhou2016thingi10k}. The left of Fig.~\ref{fig:remesh_2}, Fig.~\ref{fig:remesh_3}, and Fig.~\ref{fig:remesh_4} correspond to the first approach. The mesh structure on the octagon is constructed using the DistMesh software~\cite{persson2004simple}. The right of Fig.~\ref{fig:remesh_2}, Fig.~\ref{fig:remesh_3}, and Fig.~\ref{fig:remesh_4} correspond to the second approach. We choose a triangulated torus as the target surface and take the map from it to the canonical polygon as the target mesh structure.

\begin{figure}
\begin{subfigure}[t]{.5\textwidth}
\centering
\includegraphics[width=.9\linewidth]{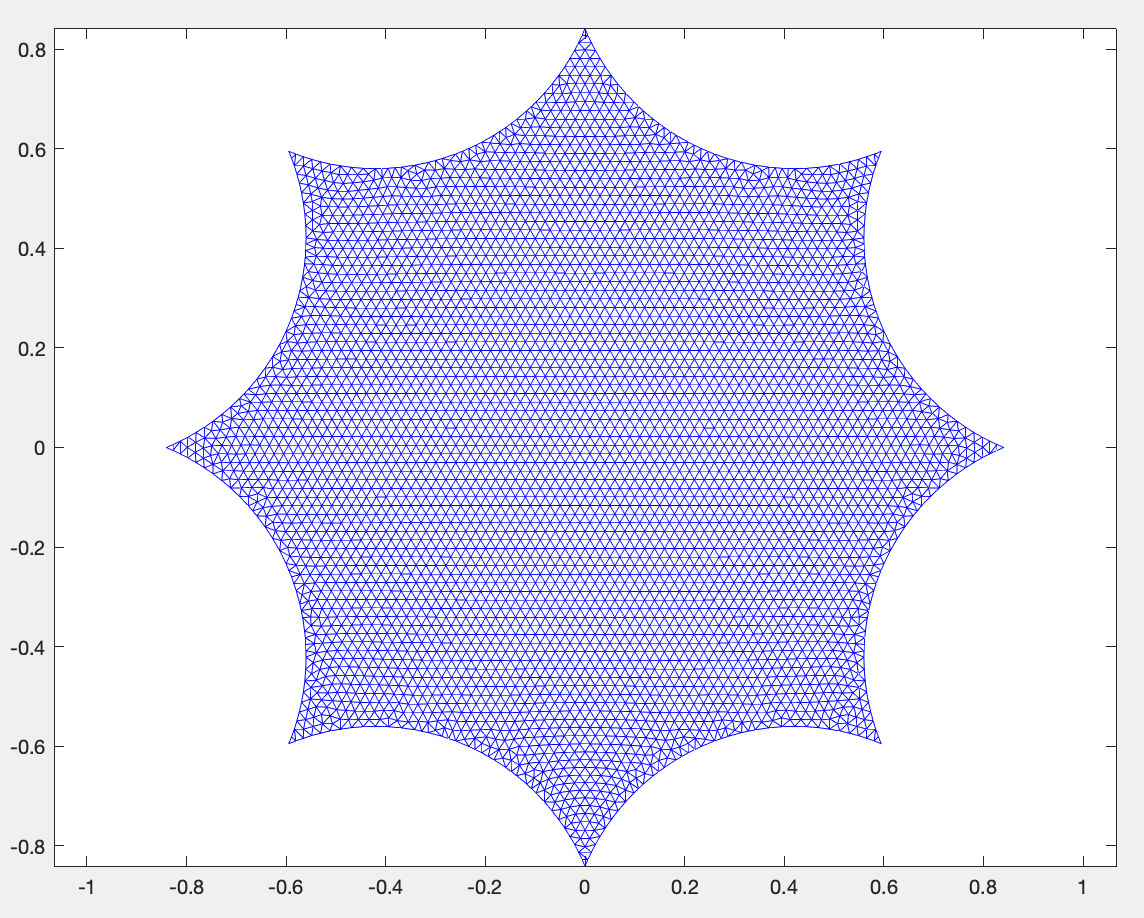}
\caption{A mesh constructed on a octagon}
\label{fig:remesh_template_1}
\end{subfigure}%
\begin{subfigure}[t]{.5\textwidth}
\centering
\includegraphics[width=.9\linewidth]{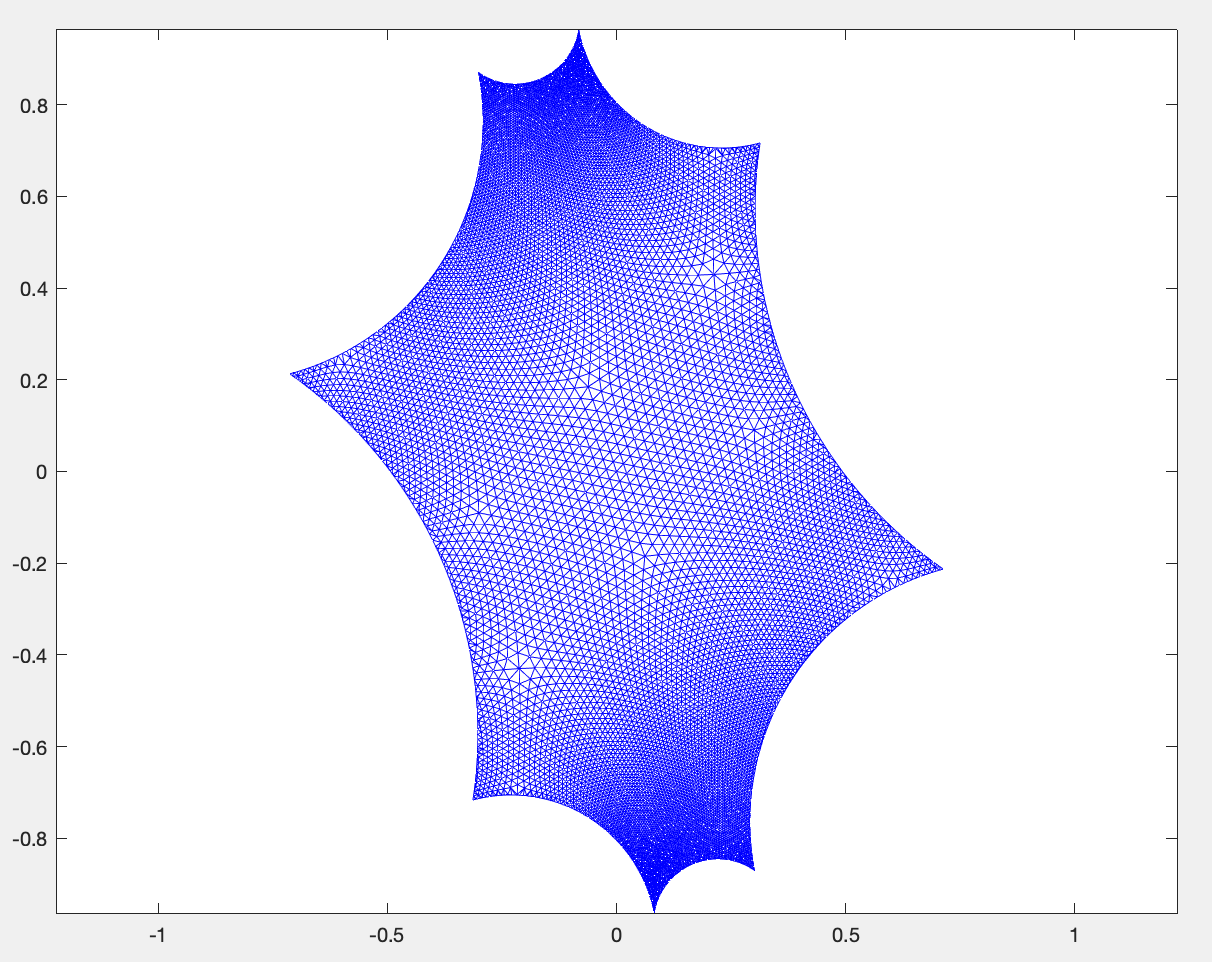}
\caption{The canonical polygon of a surface}
\label{fig:remesh_template_2}
\end{subfigure}%
\caption{The mesh structures on target surfaces}
\label{fig:remesh_2}
\end{figure}

\begin{figure}
\begin{subfigure}[t]{.5\textwidth}
\centering
\includegraphics[width=.9\linewidth]{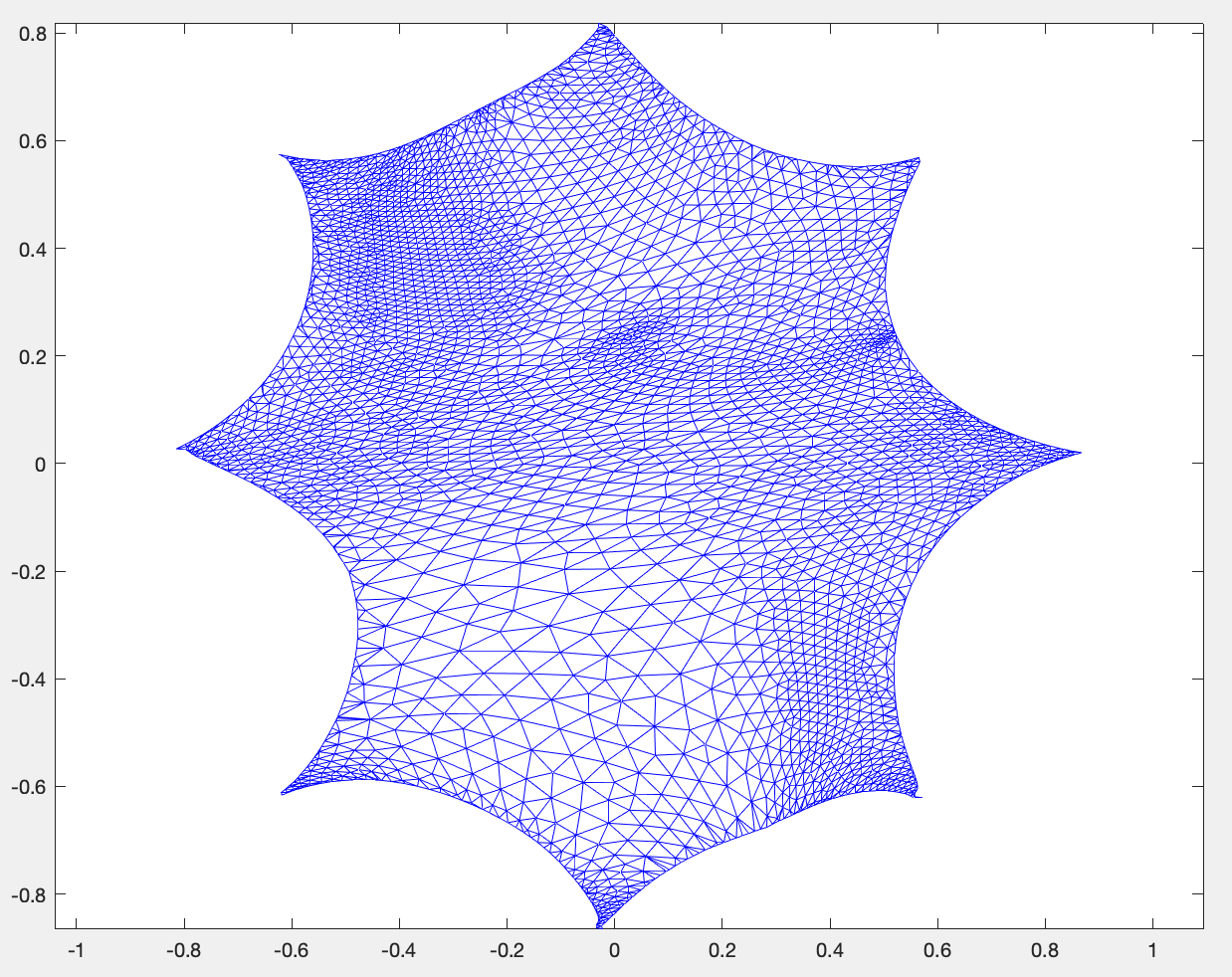}
\label{fig:remesh_map_1}
\end{subfigure}%
\begin{subfigure}[t]{.5\textwidth}
\centering
\includegraphics[width=.9\linewidth]{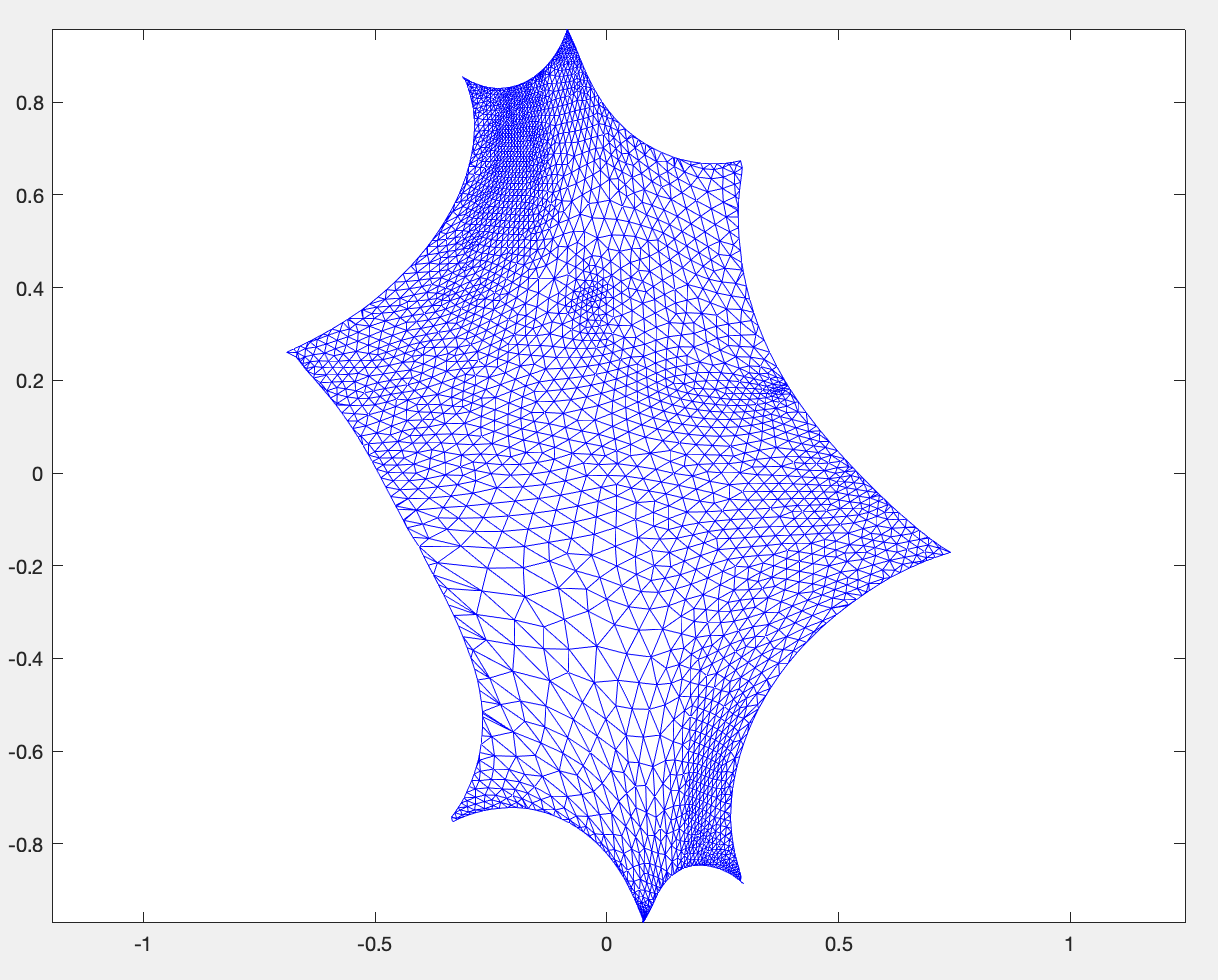}
\label{fig:remesh_map_2}
\end{subfigure}%
\caption{The harmonic maps from the input surface to the target ones}
\label{fig:remesh_3}
\end{figure}

\begin{figure}
\begin{subfigure}[t]{.494\textwidth}
\centering
\includegraphics[width=.7\linewidth]{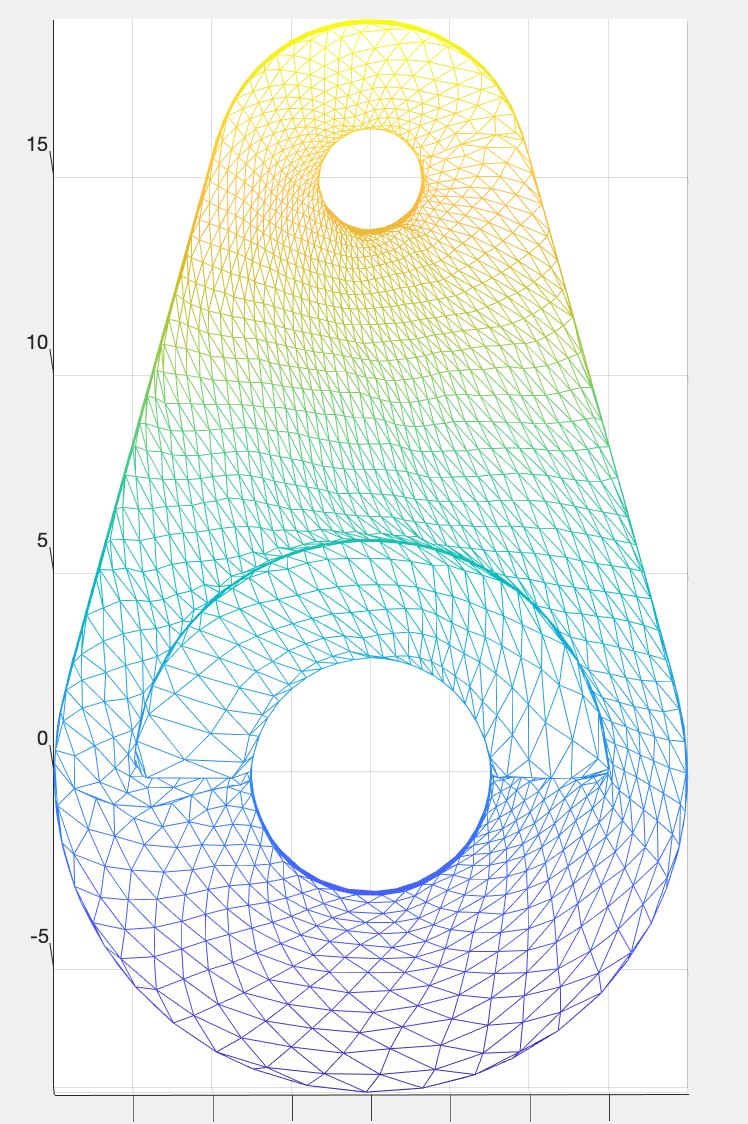}
\label{fig:remesh_result_1}
\end{subfigure}%
\begin{subfigure}[t]{.506\textwidth}
\centering
\includegraphics[width=.7\linewidth]{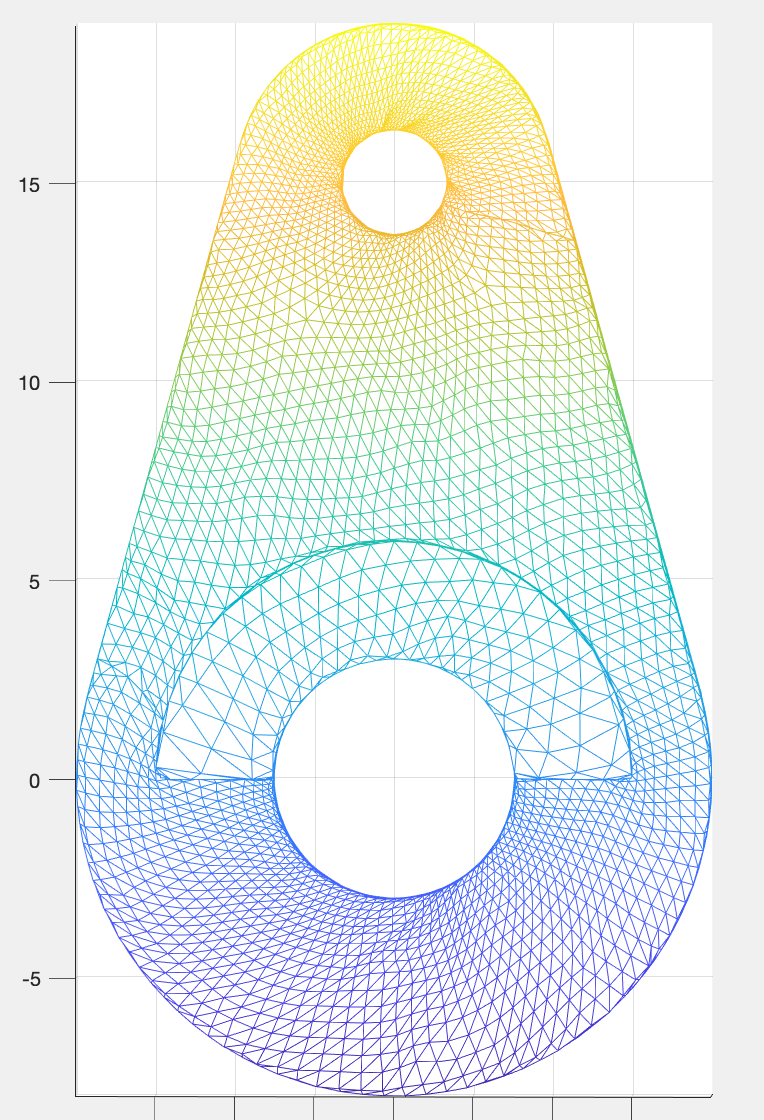}
\label{fig:remesh_result_2}
\end{subfigure}%
\caption{The two remeshed surfaces}
\label{fig:remesh_4}
\end{figure}

\bibliographystyle{siamplain}
\bibliography{harmonic}

\end{document}